\documentclass[11pt,reqno,UTF8,twoside]{amsart}
\usepackage{mathrsfs}
\usepackage{amsfonts,amssymb,amsmath,amsthm}
\usepackage{cite}
\usepackage[colorlinks=true,citecolor=red,linkcolor=blue]{hyperref}
\usepackage{titletoc}
\usepackage{geometry}
\usepackage{bm}
\usepackage{indentfirst}
\usepackage{graphicx}
\usepackage{stmaryrd}
\usepackage{tikz}
\usepackage{ulem}
\usepackage{color,soul}
\usepackage{setspace}
\usepackage[pagewise]{lineno}
\usepackage{float}
\usepackage{appendix}

\newcommand{\blackhyperref}[2]{%
  {%
    \hypersetup{linkcolor=black}%
    \hyperref[#1]{#2}%
  }%
}

\usepackage{enumitem}

\usepackage{todonotes}

\def\T {{\mathbb T}}

\def\ii {\mathrm{i}}
\def\ee {\mathrm{e}}
\def\R {\mathbb{R}}
\def\N {\mathbb{N}}
\def\Z {\mathbb{Z}}
\def\d{{\,\rm d}}
\numberwithin{equation}{section}
\newtheorem{theorem}{Theorem}[section]
\newtheorem{lemma}[theorem]{Lemma}

\newtheorem{proposition}[theorem]{Proposition}
\newtheorem{remark}[theorem]{Remark}
\theoremstyle{definition}

\newtheorem{definition}[theorem]{Definition}

\newcommand{\supp}{\operatorname{supp}}

\tikzstyle{idea} = [rectangle, rounded corners, minimum width=2cm, minimum height=1cm, text centered, draw=black, align=center]
\tikzstyle{process} = [rectangle, minimum width=3cm, minimum height=1cm, text centered, draw=black, align=center]
\tikzstyle{point} = [coordinate, on grid]
\tikzstyle{arrow} = [thick,->,>=stealth]
\tikzstyle{dasharrow} = [dashed,->,>=stealth]

\makeatletter
\@namedef{subjclassname@2020}{\textup{2020} Mathematics Subject Classification}
\makeatother

\title[Symmetry conditions for wave observability]{\large Symmetry conditions for spacetime observability of  wave equations on the torus }

\author[Jingrui Niu]{Jingrui Niu}

\author[Ming Wang]{Ming Wang}

\author[Shengquan Xiang]{Shengquan Xiang}

\address[Jingrui Niu]{Sorbonne Université, CNRS, Université Paris Cité, Inria Team CAGE, Laboratoire Jacques-Louis Lions (LJLL), F-75005 Paris, France
}
\email{jingrui.niu@sorbonne-universite.fr}

\address[Ming Wang]{School of Mathematics and Statistics, HNP-LAMA, Central South University, Chang-sha, Hunan 410083, P.R. China}
 \email{m.wang@csu.edu.cn}

\address[Shengquan  Xiang]{School of Mathematical Sciences, Peking University, 100871, Beijing, China.}
\email{shengquan.xiang@math.pku.edu.cn}

\begin{document}

	

    \vspace{5mm}
\begin{abstract}
We study observability for the one-dimensional wave equation on the torus from spacetime measurable observation sets. While the Geometric Control Condition (GCC) provides a sufficient criterion in many classical settings, it is no longer sufficient in this framework.

We construct explicit counterexamples showing the failure of observability despite the validity of GCC. This leads to the introduction of an additional symmetry condition on the observation set, referred to as the Observable Symmetry Condition (OSC).

We prove that observability holds if and only if both GCC and OSC are satisfied. We also show that unique continuation holds if and only if both OSC and a weak form of GCC are satisfied.
\end{abstract}

 \maketitle

\section{Introduction}
Let $T > 0$ and let $G \subset [0, T]\times [0, 2\pi]$ be a spacetime measurable set.
We consider the {\it observability}  of the wave equation on the observation set $G$ in the energy space  $H^1(\T)\times L^2(\T)$, that is, whether there exists a constant $C>0$ such that every solution $u$ of 
\begin{equation}\label{eq: wave-eq-0}
(\partial_t^2-\partial_x^2)u=0,\quad (u,\partial_tu)\big|_{t=0}=(u_0,u_1)\in H^1(\T)\times L^2(\T)
\end{equation}
satisfies
\begin{equation}\label{eq: wave-ob}
\|\partial_x u_0\|^2_{L^2(\T)}+\|u_1\|^2_{L^2(\T)}\leq C\iint_G|\partial_tu(t,x)|^2\d x\d t.
\end{equation}

By the classical Hilbert uniqueness method due to Lions \cite{Lions-book}, the observability is equivalent to the exact controllability of the wave equation; see, for instance, \cite{Coron-book, Tucsnak2009book}.

\subsection{Background and setting}
The study of observability and controllability for wave equations dates back to the works of Russell \cite{Russell-1} and Lions \cite{Lions-book}.

\subsubsection{Geometric control condition}
In the pioneering work of Bardos--Lebeau--Rauch \cite{BLR-gcc}, observability for wave equations was related to geometric properties of the observation region through the introduction of the geometric control condition (GCC): for an open set $\omega$, every generalized ray meets $\omega$ in finite time.

In the classical setting of cylindrical domains $(0,T)\times\omega$ with $\omega$ open, GCC is known to be necessary and sufficient for observability; see \cite{BLR-gcc,BG-97}. Related results also appear in regional observability, stabilization, optimization, inverse problems, memory kernel, and numerical studies; see, for instance,  \cite{Lebeau-1996, Burq-98, zuazua-review, BHHR2015, Trelat, Burq-Gerard-2020, Krieger-Xiang-2024, WangZhangZuazua, Dehman-Sylvain-Zuazua}.

\subsubsection{Unique continuation}
A qualitative counterpart of observability is the unique continuation property.
Classical approaches rely on analyticity arguments, such as Holmgren's theorem, as well as on pseudo-convexity conditions and Carleman estimates; see, for instance, \cite{RZ-98,Tataru,Hormander-96}. More recent developments in this direction include \cite{Laurent-Leautaud-2019,Shao}.
In the spacetime setting, unique continuation involves propagation across spacetime surfaces, which is more delicate than in cylindrical geometries.

\subsubsection{Spacetime measurable observation set} 
 In this paper, we consider the setting where $G\subset[0,T]\times [0, 2\pi]$ is a measurable set.  Related problems in spacetime measurable frameworks have been studied in various contexts; see \cite{Castro-Cindea-Munch-2014,RouLebeauAnalPDE2017,Shao,PK}. The following definition is a natural analogue of the geometric control condition in this setting.

\vspace{2mm}
\begin{enumerate}
\item[({\bf GCC})]\label{def:GCC} 
Let $T>0$. A measurable set $G\subset [0, T]\times [0, 2\pi]$ is said to satisfy GCC if there exists a constant 
$c_0 > 0$ such that for almost every $x \in [0, 2\pi]$,
\begin{align*}
    \mbox{meas}_\R\big( G \cap \{ (s,x+ s) : s\in [0, T]\}\big)&\geq c_0, \\
   \mbox{meas}_\R\big( G \cap \{ (s,x- s) : s\in [0, T]\}\big)&\geq c_0, 
\end{align*}
where we identify $(t,x+2k\pi)$ with $(t,x)$ for all $k\in\Z$.
\end{enumerate}
\vspace{2mm}

Throughout the paper, \blackhyperref{def:GCC}{({\rm GCC})}  refers to this definition.

\subsection{The counterexample}\label{Sec:int:example}

 Let $T= 2\pi$. Let $G\subset [0, 2\pi]\times [0, 2\pi]$ be given in Figure \ref{fig:GCCfail1}. The set $G$ satisfies \blackhyperref{def:GCC}{({\rm GCC})}.  However, the observability inequality \eqref{eq: wave-ob}  fails on $G$; see Section \ref{sec:example}.

\begin{figure}[htp]
    \centering
\begin{tikzpicture}[scale=1.0]

  \draw[->] (-0.2,0) -- (6.2,0) node[right] {$x$};
  \draw[->] (0,-0.2) -- (0,6.2) node[above] {$t$};

  \draw (3,0) node[below] {$\pi$} -- (3,0.1);
  \draw (6,0) node[below] {$2\pi$} -- (6,0.1);
  \draw (0,3) node[left] {$\pi$} -- (0.1,3);
  \draw (0,6) node[left] {$2\pi$} -- (0.1,6);

  \draw (0,0) -- (6,0) -- (6,6) -- (0,6) -- cycle;

\fill[blue!30] (0,0) coordinate (A1) -- (3,0) coordinate (B1) -- (1.5,1.5) coordinate (C1)-- cycle;

\node at (barycentric cs:A1=1,B1=1,C1=1) {\textcolor{white}{-1}};

\fill[red!30] (3,0) coordinate (A2) -- (6,0) coordinate (B2) -- (4.5,1.5) coordinate (C2)-- cycle;

\node at (barycentric cs:A2=1,B2=1,C2=1) {\textcolor{white}{1}};

\fill[red!30] (0,3) coordinate (A3) -- (1.5,1.5) coordinate (B3) -- (3,3) coordinate (C3)--(1.5,4.5) coordinate (D3)-- cycle;

\node at (barycentric cs:A3=1,B3=1,C3=1,D3=1) {\textcolor{white}{1}};

\fill[blue!30] (3,3) coordinate (A4) -- (4.5,1.5) coordinate (B4) -- (6,3) coordinate (C4)--(4.5,4.5) coordinate (D4)-- cycle;

\node at (barycentric cs:A4=1,B4=1,C4=1,D4=1) {\textcolor{white}{-1}};

\fill[blue!30] (0,6) coordinate (A5) -- (1.5,4.5) coordinate (B5) -- (3,6) coordinate (C5)-- cycle;

\node at (barycentric cs:A5=1,B5=1,C5=1) {\textcolor{white}{-1}};

\fill[red!30] (3,6) coordinate (A6) -- (4.5,4.5) coordinate (B6) -- (6,6) coordinate (C6)-- cycle;

\node at (barycentric cs:A6=1,B6=1,C6=1) {\textcolor{white}{1}};

  \draw[dashed] (-0.5,3.5) -- (3.5,-0.5); 
  \draw[dashed] (0,6) -- (6,0); 
  \draw[dashed] (2.5,6.5) -- (6.5,2.5); 

  \draw[dashed] (-0.5,2.5) -- (3.5,6.5); 
  \draw[dashed] (0,0) -- (6,6); 
  \draw[dashed] (2.5,-0.5) -- (6.5,3.5); 

\node at (10,4) {$G=G_{1}\cup G_{2}$};

\node at (10, 3) {$G_{1}=\{(x,t): u_\xi=u_\eta=-1\}$: blue part};

\node at (10, 2) {$G_2=\{(x,t): u_\xi=u_\eta=1\}$: red part};
  
\end{tikzpicture}
    \caption{Observability fails on $G$, though $G$ satisfies ({\rm GCC}). }
    \label{fig:GCCfail1}
\end{figure}
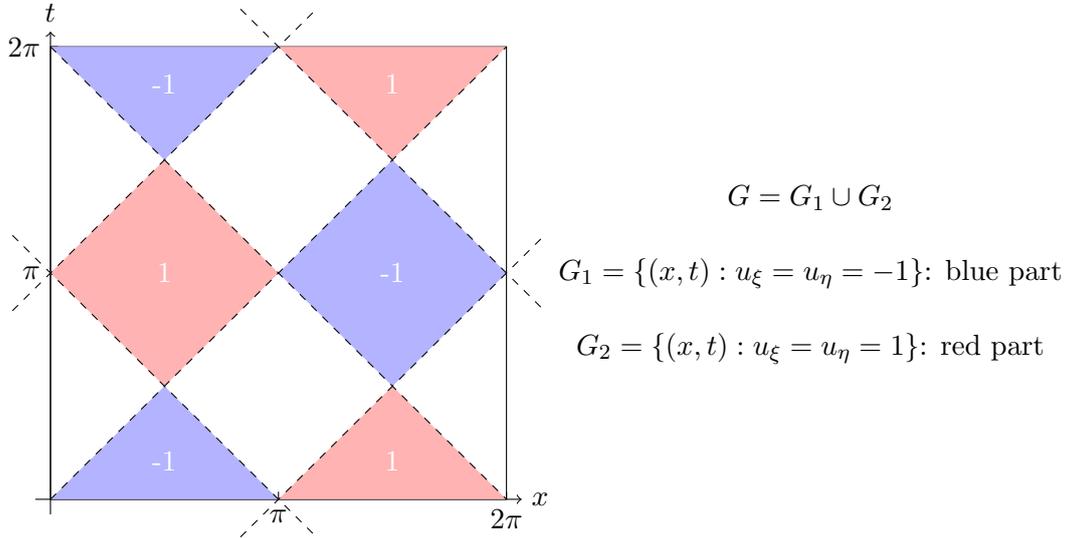

\subsection{Observable symmetry condition}

Extend the system $2\pi$-periodically to $x\in \R$,  and introduce the null coordinate:
\begin{equation}
    \xi=x+t \; \textrm{ and } \; \eta=x-t. 
\end{equation}
Under this new coordinate,  
\begin{equation*}
  2  \partial_{\xi} u= \partial_x u+ \partial_t u \; \textrm{ and } \;  2  \partial_{\eta} u= \partial_x u- \partial_t u,
\end{equation*}
and the wave equation \eqref{eq: wave-eq-0} becomes\footnote{We considered the equation on $U(\xi, \eta):= u(t, x)$, but still denote $U$ by $u$ to simplify the notation.  }
\begin{equation}
     \partial_{\xi} \partial_{\eta} u=  0. \notag
\end{equation}

For any $\xi_0\in\R$, we denote the union of lines as follows, and call it the $\xi$-characteristic,
\begin{equation}
  L_{\xi=\xi_0}:=  \bigcup_{k\in \mathbb{Z}} \{(t,x):x+t=\xi_0+ 2k\pi\}.
\end{equation}
Similarly, we denote the union of lines $\{(t,x):x- t=\eta_0+ 2k\pi, k\in \mathbb{Z}\}$ by $L_{\eta=\eta_0}$ and  call it the $\eta$-characteristic. Thus the set $G\subset [0, T]\times [0, 2\pi]$ satisfies \blackhyperref{def:GCC}{({\rm GCC})} is equivalent to
\begin{equation}\label{eq:GCC:measure}
    \exists c_0>0,   \textrm{ s.t. for } a.e. \; x\in [0, 2\pi],  \; \mbox{meas}_\R(G\cap L_{\xi= x})\geq c_0, \;\;  
  \mbox{meas}_\R(G\cap L_{\eta= x})\geq c_0. 
\end{equation} 

We further define measurable cylinders for any measurable sets $A, B\subset [0, 2\pi]$ as
\begin{equation}\label{eq: cylinder}
L_{\xi\in A}:=\bigcup_{\xi_0\in A} L_{\xi=\xi_0} \; \textrm{ and } \; L_{\eta\in B}:=\bigcup_{\eta_0\in B} L_{\eta=\eta_0}
\end{equation}
and define, for $T\geq 0$
\begin{equation}\label{eq: cylinder:T}
L_{\xi\in A}^T:=L_{\xi\in A}\cap \left([0, T]\times \R\right) \; \textrm{ and } \; L_{\eta\in B}^T:=L_{\eta\in B}\cap \left([0, T]\times \R\right).
\end{equation}

Motivated by the counterexample, we introduce the following geometric symmetry condition.
\begin{definition}
   Let $A, B\subset [0, 2\pi]$ be two measurable subsets. Let $0\leq S\leq T$. 
   A set $G\subset [0, T]\times [0, 2\pi]$ is said to be  ($A, B; S$)-{\it observable symmetric}  if
\begin{equation}\label{wave:nece:2}
 \mbox{meas}_{\R^2} \Big( [G\cap L_{\xi\in A}^S] \, \Delta \, [G\cap L_{\eta\in B}^S]\Big)=0,
\end{equation} 
where $X \Delta Y = (X\backslash Y)\cup (Y\backslash X)$ denotes the symmetric difference between $X$ and $Y$. A pair $(A, B)$ is called trivial, if $|A|= |B|= 0$ or $|A|= |B|= 2\pi$; Otherwise, it is called non-trivial.
\end{definition}
Expression \eqref{wave:nece:2} means that $G\cap L_{\xi\in A}^S = G\cap L_{\eta\in B}^S$ modulo zero measure set in $\R^2$. 
For instance, in the counterexample,  the set $G$ is ($A, B$)-observable symmetric for $A=B=(0,\pi)$ and for $A=B=(\pi,2\pi)$. 
 Note that any set $G$ is ($A, B; T$)-observable symmetric for trivial pairs. We will be concerned with {\it non-trivial} pairs $(A, B)$.

\vspace{2mm}
The first result is a new conservation law for the controlled wave equation,
\begin{equation}\label{eq: wave-eq}
(\partial_t^2-\partial_x^2)u =f \chi_G,\quad (u, \partial_t u)\big|_{t=0}\in H^1(\T)\times L^2(\T),
\end{equation}
where $G\subset [0, T]\times [0, 2\pi]$ is a measurable set and  $f\in L^2((0, T)\times \T)$ is the control function.
\begin{proposition}[Conservation law]
\label{prop-sym-conser}
Let $A,B\subset [0, 2\pi]$. Let $T>0$.  Assume that $G\subset [0, T]\times [0, 2\pi]$ is $(A, B; T)$-{\it observable symmetric}.
Let $(u,u_t)\in C([0,T];H^1(\T)\times L^2(\T))$ be a solution to the equation \eqref{eq: wave-eq} and define the quantity 
$$
I(t)= \int_{A-t}(\partial_x u+ \partial_t u)(t,x)\d x+\int_{B+t}(\partial_x u- \partial_t u)(t,x)\d x,
$$
 where $A-t = \{x-t \textrm{ mod } 2\pi: x \in A\}, B+t = \{x+t \textrm{ mod } 2\pi : x \in B\}$.
Then
$$
I(t)=I(0), \;  \forall t\in [0, T].
$$
In the case $|A|= |B|= 2\pi$, this conservation is reduced to the condition $\int_{\T} \partial_x u (t, x) \d x= 0$.
\end{proposition}

We are now ready to introduce the observable symmetry condition.
\vspace{2mm}

\begin{enumerate}
\item[({\bf OSC})]\label{def:OSC} 
Let $T>0$. A set $G\subset [0, T]\times [0, 2\pi]$ is said to satisfy the {\it observable symmetry condition } (OSC),  if for any non-trivial pair ($A, B$) the set $G$ is not ($A, B; T$)-observable symmetric. 
\end{enumerate}
\vspace{2mm}

This condition depends on $T$; we write $({\rm OSC})$ without $T$ to simplify the notation. As will be shown in Section~\ref{sec:symmetry}, the observable symmetry condition yields a {\it necessary condition} for the controllability and even for the unique continuation property of  \eqref{eq: wave-eq-0}.  

\subsection{Unique continuation property}
The second result is about the unique continuation property (UCP).  Let $T>0$ and let $G\subset [0, T]\times [0, 2\pi]$ be a measurable set.  UCP is the qualitative version of the observability and asks:
\begin{equation*}
    \mbox{(UCP) \;\;\;\;\;\;\;\;\;\;\;\;\;\; Let $u$ be a solution of \eqref{eq: wave-eq-0} and }   \partial_t u =0 \; \mbox{ a.e. in } G \;\; \Longrightarrow \;\; u \equiv \mbox{const}.  \;\;\;\;\;\;\;\;\;\;\;\;\;\;\;\;\;\;\;\;\;\;\;\;\;\;\; 
\end{equation*}

Define the following geometric assumption on the observation set.
\vspace{2mm}

\begin{enumerate}[label={}, leftmargin=6.5em, labelindent=3em]
\item[({\bf Weak  GCC})]\label{def:weakGCC} 
Let $T>0$. A measurable set $G\subset [0, T]\times [0, 2\pi]$ is said to satisfy the 
 weak GCC if for almost every 
$x\in[0, 2\pi]$,  
\begin{equation*}
 \mbox{meas}_\R(G\cap L_{\xi= x})> 0, \;\;  
  \mbox{meas}_\R(G\cap L_{\eta= x})> 0. 
\end{equation*}
\end{enumerate}

\vspace{2mm}

Compared to \blackhyperref{def:GCC}{({\rm GCC})}, it is not required a uniform positive lower bound.

\begin{theorem}\label{thm: main-ucp}
Let $T>0$ and let $G\subset [0,T]\times [0, 2\pi]$ be a measurable set. Unique continuation holds on $G$ if and only if $G$ satisfies
\blackhyperref{def:OSC}{$({\rm OSC})$} and \blackhyperref{def:weakGCC}{{\rm (weak GCC)}}. 
\end{theorem}

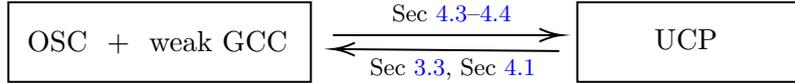
\begin{figure}[htp]
    \centering

\tikzset{every picture/.style={line width=0.75pt}} 

\begin{tikzpicture}[x=0.75pt,y=0.75pt,yscale=-1,xscale=1]

\draw   (141,130) -- (292,130) -- (292,170) -- (141,170) -- cycle ;
\draw   (425,130) -- (539,130) -- (539,170) -- (425,170) -- cycle ;
\draw    (303,146.5) -- (415,146.5) ;
\draw [shift={(417,146.5)}, rotate = 180] [color={rgb, 255:red, 0; green, 0; blue, 0 }  ][line width=0.75]    (10.93,-3.29) .. controls (6.95,-1.4) and (3.31,-0.3) .. (0,0) .. controls (3.31,0.3) and (6.95,1.4) .. (10.93,3.29)   ;
\draw    (416,154.5) -- (305,153.52) ;
\draw [shift={(303,153.5)}, rotate = 0.51] [color={rgb, 255:red, 0; green, 0; blue, 0 }  ][line width=0.75]    (10.93,-3.29) .. controls (6.95,-1.4) and (3.31,-0.3) .. (0,0) .. controls (3.31,0.3) and (6.95,1.4) .. (10.93,3.29)   ;

\draw (149,143) node [anchor=north west][inner sep=0.75pt]   [align=left] {\blackhyperref{def:OSC}{${\rm OSC}$} \ + \ \blackhyperref{def:weakGCC}{{\rm weak GCC}}};
\draw (462,142) node [anchor=north west][inner sep=0.75pt]   [align=left] {UCP};
\draw (331,130) node [anchor=north west][inner sep=0.75pt]   [align=left] {{\footnotesize Sec \ref{sec: OSC+GCC>UCP}--\ref{sec: OSC+WGCC>UCP}}};
\draw (320,157) node [anchor=north west][inner sep=0.75pt]   [align=left] {{\footnotesize Sec \ref{sec:neosctoucp}, Sec \ref{sec: ucp-wgcc}}};

\end{tikzpicture}

\end{figure}

\subsection{Observability}\label{Sec:int:obs}
Finally, we obtain the following result.
\begin{theorem}\label{thm: main-ob}
Let $T>0$ and let $G\subset [0, T]\times [0, 2\pi]$ be a measurable set. Observability inequality \eqref{eq: wave-ob}  holds on $G$ if and only if $G$ satisfies  \blackhyperref{def:OSC}{$({\rm OSC})$} and \blackhyperref{def:GCC}{{\rm (GCC)}}. 
\end{theorem}

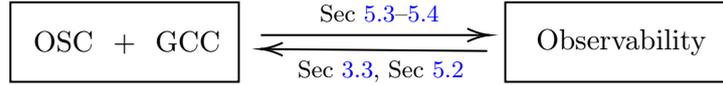
\begin{figure}[htp]
    \centering

\tikzset{every picture/.style={line width=0.75pt}} 

\begin{tikzpicture}[x=0.75pt,y=0.75pt,yscale=-1,xscale=1]

\draw   (178,130) -- (292,130) -- (292,170) -- (178,170) -- cycle ;
\draw   (425,130) -- (539,130) -- (539,170) -- (425,170) -- cycle ;
\draw    (303,146.5) -- (415,146.5) ;
\draw [shift={(417,146.5)}, rotate = 180] [color={rgb, 255:red, 0; green, 0; blue, 0 }  ][line width=0.75]    (10.93,-3.29) .. controls (6.95,-1.4) and (3.31,-0.3) .. (0,0) .. controls (3.31,0.3) and (6.95,1.4) .. (10.93,3.29)   ;
\draw    (416,154.5) -- (305,153.52) ;
\draw [shift={(303,153.5)}, rotate = 0.51] [color={rgb, 255:red, 0; green, 0; blue, 0 }  ][line width=0.75]    (10.93,-3.29) .. controls (6.95,-1.4) and (3.31,-0.3) .. (0,0) .. controls (3.31,0.3) and (6.95,1.4) .. (10.93,3.29)   ;

\draw (188,143) node [anchor=north west][inner sep=0.75pt]   [align=left] {\blackhyperref{def:OSC}{${\rm OSC}$} \ + \ \blackhyperref{def:GCC}{{\rm GCC}}};
\draw (439,142) node [anchor=north west][inner sep=0.75pt]   [align=left] {Observability};
\draw (331,130) node [anchor=north west][inner sep=0.75pt]   [align=left] {{\footnotesize Sec \ref{sec: weak-ob-sec}--\ref{sec:OSCplusGCCimplOBS}}};
\draw (320,158) node [anchor=north west][inner sep=0.75pt]   [align=left] {{\footnotesize Sec \ref{sec:neosctoucp}, Sec \ref{sec:obimplyGCC}}};

\end{tikzpicture}
\end{figure}
\vspace{2mm}

\noindent {\it Organization of the paper.}  Sections~\ref{sec:example}--\ref{sec:sharp:obser} correspond, respectively, to Sections~\ref{Sec:int:example}--\ref{Sec:int:obs}.

\section{The counterexample}\label{sec:example}
Choose an initial state $(u_0, u_1)\in \dot H^1(\T)\times L^2(\T)$ such that
\begin{gather*}
     \partial_{\xi}u|_{t= 0, x\in (0, \pi)}= -1,   \; \; \partial_{\xi}u|_{t= 0, x\in (\pi,2\pi)}= 1, \\
       \partial_{\eta}u|_{t= 0, x\in (0, \pi)}= -1,  \;  \;\partial_{\eta}u|_{t= 0, x\in (\pi, 2\pi)}= 1.
\end{gather*}
With this choice, one has
$$
\int_\T \partial_x u(t=0, x) \d x=\int_\T (\partial_\xi u+\partial_\eta u) (t=0, x) \d x=0.
$$
There is a unique solution to the wave equation \eqref{eq: wave-eq-0} on the domain $(t, x)\in (0, 2\pi)\times \T$. 
\vspace{2mm}

Define 
\begin{equation*}
    G:= \{(\xi, \eta); \partial_{\xi}u= \partial_{\eta}u= -1, \textrm{ or }   \partial_{\xi}u= \partial_{\eta}u= 1\}.
\end{equation*}
This region is divided into two parts, $G= G_1\cup G_2$ as shown in Figure \ref{fig:GCCfail1}:
\begin{align*}
    G_1:= \{(\xi, \eta); \partial_{\xi}u= \partial_{\eta}u= -1\}, \;\;\;
     G_2:= \{(\xi, \eta); \partial_{\xi}u= \partial_{\eta}u= 1\}.
\end{align*}

One checks that $G$ satisfies \blackhyperref{def:GCC}{({\rm GCC})} with $c_0= \sqrt{2}\pi$. 
However
$$
\iint_G|\partial_t u|^2\d x \d t = \iint_G|\partial_\xi u-\partial_\eta u|^2 \d \xi \d \eta=0
$$
and the corresponding initial data satisfies
$$
\|\partial_xu_0\|^2_{L^2(\T)}+\|u_1\|^2_{L^2(\T)}= 2\left(\|\partial_{\xi}u|_{t= 0, x\in \T}\|^2_{L^2(\T)}+\|\partial_{\eta}u|_{t= 0, x\in \T}\|^2_{L^2(\T)}\right)=8\pi.
$$
Therefore, the observability inequality \eqref{eq: wave-ob} fails on $G$.

\begin{remark}[A variant]
Let $\tilde G_1, \tilde G_2$ be the blue and red parts in Figure \ref{fig:GCCfail2}, and let $\tilde G=\tilde G_1\cup \tilde G_2$. Then $\tilde G$ satisfies  \blackhyperref{def:GCC}{({\rm GCC})}, but the observability inequality fails on $\tilde G$.     
\end{remark} 

\begin{figure}[htp]
    \centering
    \begin{tikzpicture}[scale=1.0]

  \draw[->] (-0.3,0) -- (6.3,0) node[right] {$x$};
  \draw[->] (0,-0.3) -- (0,6.3) node[above] {$t$};

  \draw (2,0) node[below] {$\frac{2\pi}{3}$} -- (2,0.1);
  \draw (4,0) node[below] {$\frac{4\pi}{3}$} -- (4,0.1);
 \draw (6,0) node[below] {$2\pi$} -- (6,0.1); 
  \draw (0,2) node[left] {$\frac{2\pi}{3}$} -- (0.1,2);
  \draw (0,4) node[left] {$\frac{4\pi}{3}$} -- (0.1,4);
\draw (0,6) node[left] {$2\pi$} -- (0.1,6);  
  
  \draw (0,0) -- (6,0) -- (6,6) -- (0,6) -- cycle;

  \draw[dashed] (-0.5,2.5) -- (2.5,-0.5); 
  \draw[dashed] (-0.5,4.5) -- (4.5,-0.5); 
  \draw[dashed] (1.5,6.5) -- (6.5,1.5);  
  \draw[dashed] (3.5,6.5) -- (6.5,3.5); 
  
  \draw[dashed] (1.5,-0.5) -- (6.5,4.5); 
  \draw[dashed] (3.5,-0.5) -- (6.5,2.5); 
 \draw[dashed] (-0.5,1.5) -- (4.5,6.5); 
 \draw[dashed] (-0.5,3.5) -- (2.5,6.5); 

\fill[red!30] (0,0) coordinate (A1) -- (2,0) coordinate (B1) -- (0,2) coordinate (C1)-- cycle;

\node at (barycentric cs:A1=1,B1=1,C1=1) {\textcolor{white}{1}};

\fill[blue!30] (2,0) coordinate (A2) -- (4,0) coordinate (B2) -- (3,1) coordinate (C2)-- cycle;

\node at (barycentric cs:A2=1,B2=1,C2=1) {\textcolor{white}{-2}};

\fill[red!30] (4,0) coordinate (A3) -- (6,0) coordinate (B3) -- (6,2) coordinate (C3)-- cycle;

\node at (barycentric cs:A3=1,B3=1,C3=1) {\textcolor{white}{1}};

\fill[blue!30] (0,2) coordinate (A4) -- (1,3) coordinate (B4) -- (0,4) coordinate (C4)-- cycle;

\node at (barycentric cs:A4=1,B4=1,C4=1) {\textcolor{white}{-2}};

\fill[red!30] (1,3) coordinate (A5) -- (3,1) coordinate (B5) -- (5,3) coordinate (C5)-- (3,5) coordinate (D5) --cycle;

\node at (barycentric cs:A5=1,B5=1,C5=1,D5=1) {\textcolor{white}{1}};

\fill[blue!30] (5,3) coordinate (A6) -- (6,2) coordinate (B6) -- (6,4) coordinate (C6) --cycle;

\node at (barycentric cs:A6=1,B6=1,C6=1) {\textcolor{white}{-2}};

\fill[red!30] (0,4) coordinate (A7) -- (2,6) coordinate (B7) -- (0,6) coordinate (C7) --cycle;

\node at (barycentric cs:A7=1,B7=1,C7=1) {\textcolor{white}{1}};

\fill[blue!30] (2,6) coordinate (A8) -- (3,5) coordinate (B8) -- (4,6) coordinate (C8) --cycle;

\node at (barycentric cs:A8=1,B8=1,C8=1) {\textcolor{white}{-2}};

\fill[red!30] (4,6) coordinate (A9) -- (6,4) coordinate (B9) -- (6,6) coordinate (C9) --cycle;

\node at (barycentric cs:A9=1,B9=1,C9=1) {\textcolor{white}{1}};

\node at (10,4) {$\tilde G=\tilde G_1\cup \tilde G_2$};

\node at (10, 3) {$\tilde G_1=\{(x,t): u_\xi=u_\eta=-2\}$: blue part};

\node at (10, 2) {$\tilde G_2=\{(x,t): u_\xi=u_\eta=1\}$: red part};

  \draw[->] (-0.3,0) -- (6.3,0) node[right] {$x$};
  \draw[->] (0,-0.3) -- (0,6.3) node[above] {$t$};
\end{tikzpicture}
     \caption{}
   \label{fig:GCCfail2}
\end{figure}

\section{Observable symmetry condition}\label{sec:symmetry}

In this section, we prove Proposition \ref{prop-sym-conser}, and further show that \blackhyperref{def:OSC}{$({\rm OSC})$} provides a necessary condition for controllability and unique continuation. 

\subsection{Conservation law}  
We first present the following lemma; its proof is left in Appendix \ref{sec:app:integral}.
\begin{lemma}\label{lem-Green-3}
Let $T>0$ and  let  $G\subset [0, T]\times [0, 2\pi]$ be a measurable set. Let $A,B \subset [0,2\pi]$ be two measurable sets. Let $u$ be a solution of \eqref{eq: wave-eq}. Define $U=\partial_x u+ \partial_t u, V=\partial_x u- \partial_t u$.  Then    
\begin{align*}
 \int_A U(x,0)\d x -\int_{A-T}U(x,T)\d x=-\iint_{G\cap L_{\xi\in A}^T} f \d x \d t,  
\end{align*}
\begin{align*}
\int_{B+T}V(x,T)\d x -\int_BV(x,0)\d x=-\iint_{G\cap L_{\eta\in B}^T} f\d x \d t,   
\end{align*}
 where $A-T = \{x-T \textrm{ mod } 2\pi: x \in A\}, B+ T = \{x+ T \textrm{ mod } 2\pi : x \in B\}$.
\end{lemma}

\begin{proof}[Proof of Proposition \ref{prop-sym-conser}]
We proceed the proof in two steps. 
\vspace{1mm}

{\it Step 1.} We show that $I(T)=I(0)$. Since $G$ is ($A, B; T$)-observable symmetric, 
 $G\cap L_{\xi\in A}^T$ equals to $G\cap L_{\eta\in B}^T$ up to a zero measure set. Thus for any $f$ we have
$$
\iint_{G\cap L_{\xi\in A}^T} f \d x \d t = \iint_{G\cap L_{\eta\in B}^T} f \d x \d t.
$$
This, together with Lemma \ref{lem-Green-3}, gives
\[
\int_A U(x,0)\d x -\int_{A-T}U(x,T)\d x=
\int_{B+T}V(x,T)\d x -\int_BV(x,0)\d x,
\]
which implies
$$
I(T)= \int_{A-T}U(x,T)\d x+
\int_{B+T}V(x,T)\d x=\int_{A}U(x,0)\d x+
\int_{B}V(x,0)\d x= I(0).
$$
\vspace{1mm}

{\it Step 2.}  Fix $S\in (0,T)$. 
Observe that $G$ is ($A, B; S$)-observable symmetric.
Indeed, set \(X = G \cap L_{\xi \in A}^T\), \(Y = G \cap L_{\eta \in B}^T\), \(Z = [0, S] \times \R\),  then
\[
[G \cap L_{\xi \in A}^S] \, \Delta \, [G \cap L_{\eta \in B}^S] 
= (X \cap Z) \, \Delta \, (Y \cap Z) 
= (X \, \Delta \, Y) \cap Z.
\]
which has zero measure, since $(X \, \Delta \, Y)$ does.
Performing {\it Step 1} again by replacing $T$ by $S$, we obtain  $I(S)=I(0)$.
\end{proof}

\subsection{Necessary for  controllability}\label{sec-osc-nece}
Recall the definition of the {\it exact controllability} of the equation \eqref{eq: wave-eq}, that is,  there exists a constant $C>0$ such that for any given initial and target states $(u_0, u_1)$ and $(\tilde u_0, \tilde u_1)$ in $ H^1(\T)\times L^2(\T)$, there exists a control function $f\in L^2((0, T)\times \T)$ 
such that the solution of \eqref{eq: wave-eq} with  $(u, \partial_t u)\big|_{t=0}=(u_0,u_1)$
satisfies $(u, \partial_t u)\big|_{t=T}=(\tilde u_0, \tilde u_1)$.

Based on the conservation law, one has the following. 
\begin{proposition}[Controllability implies OSC]\label{prop-control-sym-nece}
Let $T>0$ and let $G\subset [0,T]\times [0, 2\pi]$ be a measurable set. If the system \eqref{eq: wave-eq} is exactly controllable, then $G$ satisfies \blackhyperref{def:OSC}{$({\rm OSC})$}. 
\end{proposition}
\begin{proof}
We argue by contradiction.
Suppose that for some non-trivial pair $(A, B)$, the set $G$ is ($A, B; T$)-observable symmetric. 
According to Proposition \ref{prop-sym-conser},  we know that for any initial state $(u_0,u_1)\in  H^1(\T)\times L^2(\T)$ and for any force $f\in L^2((0, T)\times \T)$ one has a conserved quantity
$I(t)=I(0)$ for every $t\in [0, T]$, 
where 
$$
I(t)= \int_{A-t}(\partial_x u+ \partial_t u)(t,x)\d x+\int_{B+t}(\partial_x u- \partial_t u)(t,x)\d x.
$$
Since  $(A, B)$ is non-tirival, one can find an initial state such that $I(0)$ is not zero. Hence, the system is not null controllable, and therefore, not exactly controllable. 
\end{proof}

\subsection{Necessary for unique continuation}\label{sec:neosctoucp}
We further show the following.
\begin{proposition}[UCP implies OSC]\label{prop:UCtoOSC}
Let $T>0$ and let $G\subset [0,T]\times [0, 2\pi]$ be a measurable set.
Assume that the following holds
$$
\mbox{ Let u be a solution of \eqref{eq: wave-eq-0} and }  \partial_t u =0 \; \mbox{ a.e. in } G \;\; \Longrightarrow \;\; u \equiv \mbox{const}.
$$
Then $G$ satisfies \blackhyperref{def:OSC}{$({\rm OSC})$}.  
\end{proposition}
\begin{proof}
We argue by contradiction. Suppose that for some non-trivial pair $(A, B)$, the set $G$ is ($A, B; T$)-observable symmetric. 
Then, as shown in Figure \ref{fig:ucp-to-OSC}, 
\begin{align}\label{equ-ucp-sym-1}
G \subset \big( \left(L_{\xi\in A}\cap L_{\eta\in B}\right)\cup \left(L_{\xi\in \T\setminus A}\cap L_{\eta\in \T\setminus B}\right)\big) \cap \left([0, T]\times [0, 2\pi]\right).
\end{align} 
\begin{figure}[htp]
    \centering
\tikzset{every picture/.style={line width=0.75pt}} 

\begin{tikzpicture}[x=0.75pt,y=0.75pt,yscale=-1,xscale=1]

\draw  [line width=0.5pt, fill={blue!30}  ,fill opacity=1 ] (162.61,59) -- (424.61,59) -- (424.61,214) -- (162.61,214) -- cycle ;

\draw  [dashed, line width=0.1pt, fill={rgb, 255:red, 255; green, 255; blue, 255 },fill opacity=1 ] (346.83,59) -- (380.23,59) -- (240.39,214) -- (207,214) -- cycle ;
\draw  [dashed, line width=0.1pt, fill={rgb, 255:red, 255; green, 255; blue, 255 }  ,fill opacity=1 ] (223.77,59) -- (187.77,59) -- (355,214.1) -- (391,213.82) -- cycle ;
\draw   (207.6,214.2) .. controls (207.65,218.79) and (209.97,221.05) .. (214.56,220.99) -- (214.56,220.99) .. controls (221.11,220.92) and (224.41,223.17) .. (224.47,227.75) .. controls (224.41,223.17) and (227.66,220.84) .. (234.21,220.76)(231.26,220.79) -- (234.21,220.76) .. controls (238.79,220.7) and (241.05,218.38) .. (241,213.8) ;
\draw   (354.4,214.2) .. controls (354.5,219.27) and (356.88,221.55) .. (361.55,221.45) -- (362.96,221.41) .. controls (369.62,221.27) and (373,223.53) .. (373.11,228.2) .. controls (373,223.53) and (376.28,221.13) .. (382.95,220.98)(379.95,221.04) -- (384.35,220.95) .. controls (389.02,220.85) and (391.3,218.47) .. (391.2,213.8) ;
\draw  [dashed, line width=0.1pt, fill={red!30}  ,fill opacity=1 ] (290.8,120.9) -- (309.1,138) -- (292.8,156.7) -- (274.15,139.25) -- (290.8,120.9) -- cycle ;

\draw (215.49,227.53) node [anchor=north west][inner sep=0.75pt]  [font=\small,rotate=-0.72]  {$B$};
\draw (364.69,228.33) node [anchor=north west][inner sep=0.75pt]  [font=\small,rotate=-0.72]  {$A$};
\draw (319,84.2) node [anchor=north west][inner sep=0.75pt]  [font=\footnotesize]  {$L_{\eta \in B}$};
\draw (224.6,82.6) node [anchor=north west][inner sep=0.75pt]  [font=\footnotesize]  {$L_{\xi \in A}$};
\draw (367.6,125.2) node [anchor=north west][inner sep=0.75pt]  [font=\small]  {$G$};

\draw  [line width=0.5pt ] (162.61,59) -- (424.61,59); 
\draw  [line width=0.5pt] (424.61,214) -- (162.61,214);
\end{tikzpicture}
 
    \caption{$G$ is a subset of the blue and red part.}
    \label{fig:ucp-to-OSC}
\end{figure}
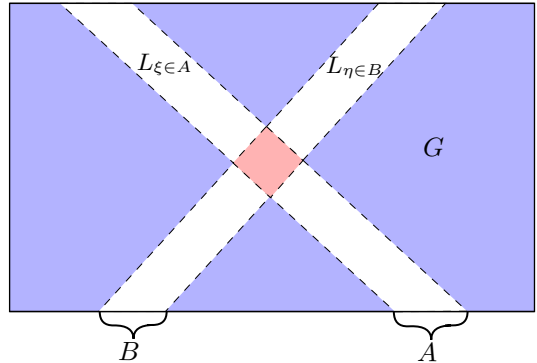

We will construct a non-constant solution such that $\partial_t u= 0$ a.e. in $G$.  Choose the initial state $(u_0, u_1)\in \dot H^1(\T)\times L^2(\T)$ such that
\begin{align*}
\partial_{\xi} u|_{t=0}=\chi_{A}+a\chi_{\T\backslash A}, \quad \partial_{\eta} u|_{t=0}=\chi_B+a\chi_{\T\backslash B},  
\end{align*}
where $a$ is given by
\begin{align*} 
a=-\frac{\mbox{meas}_\R (A)+\mbox{meas}_\R (B)}{4\pi-(\mbox{meas}_\R (A)+\mbox{meas}_\R (B))} 
\end{align*}
This choice verifies that
$$
\int_\T \partial_x u_{0} \d x =\int_\T (\partial_\xi u + \partial_\eta u)|_{t=0}\d x=0.
$$
Solving the wave equation,
\begin{align*}
    \partial_\xi u=1 \mbox{ on } L_{\xi\in A}, &\quad \partial_\xi u=a \mbox{ on } L_{\xi\in \T\backslash A}, \\
    \partial_\eta u=1  \mbox{ on } L_{\eta\in B}, &\quad \partial_\eta u=a  \mbox{ on } L_{\eta\in \T\backslash B}.
\end{align*}

It follows that $u$ is not a constant, and satisfies 
\begin{align*}
 \partial_t u= \partial_\xi u-\partial_\eta u = 0    \mbox{ on } (L_{\xi\in A}\cap L_{\eta\in B})\cup  (L_{\xi\in \T\backslash A}\cap L_{\eta\in \T\backslash B}).
\end{align*}
This non-constant solution $u$ satisfies $\partial_t u= 0$ in $G$. Therefore, UCP does not holds on $G$.
\end{proof}

\section{Unique continuation property}
This section is devoted to the proof of Theorem \ref{thm: main-ucp}. Section \ref{sec: ucp-wgcc} shows that \blackhyperref{def:weakGCC}{({\rm weak GCC})} is necessary for unique continuation. The sufficiency is proved in two steps. First, we obtain unique continuation under \blackhyperref{def:OSC}{$({\rm OSC})$} and \blackhyperref{def:GCC}{({\rm GCC})} in Section \ref{sec: OSC+GCC>UCP}. Then, in Section \ref{sec: OSC+WGCC>UCP} we further prove unique continuation under \blackhyperref{def:OSC}{$({\rm OSC})$} and \blackhyperref{def:weakGCC}{({\rm weak GCC})}. Although the latter already suffices to build the sufficient part, the former offers a simpler setting to present the ideas.

\subsection{Unique continuation implies weak GCC}\label{sec: ucp-wgcc}
Throughout this section, since the value of $T>0$ is fixed, we simply denote
\begin{align*}
    L_{\xi= \xi_0}\cap \left([0, T]\times \R\right), \;  L_{\xi\in A}\cap \left([0, T]\times \R\right), \;  L_{\eta= \eta_0}\cap \left([0, T]\times \R\right),  \;  L_{\eta\in B}\cap \left([0, T]\times \R\right),
\end{align*}
by  $L_{\xi= \xi_0}, L_{\xi\in A}, L_{\eta= \eta_0}$, and $L_{\eta\in B}$.
\vspace{2mm}

If the initial state $(u_0, u_1)$ is smooth enough, then the unique solution to \eqref{eq: wave-eq-0} satisfies that $\partial_{\xi}u$ is invariant along the $\xi$-characteristics. In the setting with $(u_0, u_1)\in  H^1(\T)\times L^2(\T)$, one has the following result, while its proof is left in Section \ref{proof-lem-char}.
\begin{lemma}\label{lem-char}
Let $u$ be a solution to \eqref{eq: wave-eq-0}. Then for almost every $\xi_0\in [0, 2\pi]$, $\partial_\xi u$ is a constant almost everywhere on the $\xi$-characteristic $L_{\xi= \xi_0}$ (resp. for $\partial_{\eta} u$ on $\eta$-characteristic $L_{\eta= \eta_0}$).  
\end{lemma}

\begin{proposition}[Unique continuation implies weak GCC]\label{prop:ucp-weakGCC}
Let $T>0$ and $G\subset [0, T]\times [0, 2\pi]$ be a measurable set.  Assume that the following unique continuation property holds
$$
\mbox{ Let u be a solution of \eqref{eq: wave-eq-0} and }   \partial_t u =0 \mbox{ a.e. in } G \;\; \Longrightarrow \;\; u \equiv \mbox{const.} 
$$
Then the set $G$ satisfies the \blackhyperref{def:weakGCC}{{\rm (weak GCC)}}.  
\end{proposition}
\begin{proof}
We argue by contradiction.  Suppose that $G$ does not satisfy the \blackhyperref{def:weakGCC}{({\rm weak GCC})},  then there exists a subset $A\subset [0, 2\pi]$ with $\mbox{meas}_\R(A)>0$ such that
\begin{align*}
\mbox{meas}_\R(G\cap L_{\xi\in A})=0 \quad \mbox{ or } \quad 
  \mbox{meas}_\R(G\cap L_{\eta\in A})=0.
\end{align*}
Without loss of generality, we assume that 
\begin{align}\label{equ-ucp-wgcc-1}
\mbox{meas}_\R(G\cap L_{\xi\in A})=0.   
\end{align}
Since $A$ has positive measure, there is a non-zero function $\phi$ such that
\begin{align}\label{equ-ucp-wgcc-2}
  \int_A \phi(x)\d x=0, \quad \int_A|\phi| \d x=\mbox{meas}_{\R}(A)>0.  
\end{align}
Choose an initial state $(u_0, u_1)\in \dot H^1(\T)\times L^2(\T)$ such that 
$ \partial_x u_{0}= u_1= \chi_A \,\phi.
$
This is possible since $\int_\T \partial_x u_{0}\d x =\int_A\phi \d x =0$. Then
$$
2\partial_\xi u|_{t=0}= \partial_x u_{0}+u_1=2\chi_A \, \phi, \quad 2\partial_\eta u|_{t=0}=\partial_x u_{0}- u_1= 0.
$$
Solving the wave equation \eqref{eq: wave-eq-0} with this initial state  $(u_0, u_1)$, it follows that
$$
\partial_\xi u=0 \quad \mbox{ in } L_{\xi\in \T\backslash A},
$$
and
$$
\partial_\eta u= 0 \quad \mbox{ in } [0,T]\times \T.
$$
Thus
$$
\partial_t u=\partial_\xi u-\partial_\eta u= 0 \quad \mbox{ in } L_{\xi\in \T\backslash A}.
$$
Thanks to \eqref{equ-ucp-wgcc-1}, we see $G\subset L_{\xi\in \T\backslash A}$, and thus $\partial_t u=0$ in $G$. By the UCP assumption, we know that $u$ is a constant. Therefore, $0 = \partial_\xi u|_{t=0}= \chi_A\,\phi$, which leads to a contradiction with \eqref{equ-ucp-wgcc-2}.
\end{proof}

\subsection{Symmetric function pairs}

Introduce the following notions, which will be used in the proof of
the unique continuation property.
\begin{definition}\label{def:decom:pair}
A family $\{(A_k, B_k)\}_{k= 1}^K$ is called a {\it decomposition pair} if the following
conditions are satisfied:
\begin{gather*}
A_k, B_k\subset [0, 2\pi] \textrm{ are measurable sets of positive Lebesgue measure for each $k=1,..., K$, }\\ 
\{A_k\}_{k= 1}^K \textrm{ are pairwise disjoint, } \quad 
\{B_k\}_{k= 1}^K \textrm{ are pairwise disjoint, }\\
\mbox{meas}_{\R}\big([0,2\pi]\setminus\cup_{1\leq k\leq K} A_k\big)= \mbox{meas}_{\R}\big([0,2\pi]\setminus\cup_{1\leq k\leq K} B_k\big)= 0.
\end{gather*}
\end{definition}

\begin{definition}\label{def:classSc}
A pair of functions
$(f,g)$ is called a {\it symmetric  function pair} if there exist a decomposition pair $\{(A_k, B_k)\}_{k= 1}^K$ and  a family of pairwise distinct real numbers $\{s_k\}_{k= 1}^K$ such that 
\begin{equation*}
 f=\sum_{k=1}^K s_k\, \chi_{A_k}, \qquad g=\sum_{k= 1}^K s_k\, \chi_{B_k} \quad \mbox{ for a.e. } x\in \T,
\end{equation*}
and
$
\int_\T (f+g)\d x= 0.
$
The set of all symmetric function pairs is denoted by $\mathcal{S}^2_c$.
\end{definition}
Note that $f$ and $g$ take the same values $s_k$ on the respective sets $A_k$ and $B_k$, each of positive measure. 
We also introduce weaker notions.

\begin{definition}\label{def:weakdecpair} Let $I$ be a countable index set (possibly finite). A family $\{(A_k, B_k)\}_{k\in I}$ is called a {\it weak decomposition pair} if the following
conditions are satisfied:
\begin{gather*}
A_k, B_k\subset [0, 2\pi] \textrm{ measurable sets of positive Lebesgue measure for each $k\in I$, }\\ 
\{A_k\}_{k\in I} \textrm{ are pairwise disjoint, } \quad 
\{B_k\}_{k\in I} \textrm{ are pairwise disjoint, }\\
\mbox{meas}_{\R}\big([0,2\pi]\setminus\cup_{k\in I} A_k\big)= \mbox{meas}_{\R}\big([0,2\pi]\setminus\cup_{k\in I} B_k\big)= 0.
\end{gather*}
\end{definition}

\begin{definition}\label{def:classS}
A pair of functions
$(f,g)$ is called a {\it weak symmetric  function pair} if there exist a weak decomposition pair $\{(A_k, B_k)\}_{k\in I}$ and  a family of pairwise distinct real numbers $\{s_k\}_{k\in I}$ such that 
\begin{equation*}
 f=\sum_{k\in I} s_k\, \chi_{A_k}, \qquad g=\sum_{k\in I} s_k\, \chi_{B_k} \quad \mbox{ in } L^2(\T),
\end{equation*}
and
$
\int_\T (f+g)\d x= 0.
$
The set of all weak symmetric function pairs is denoted by $\mathcal{S}^2$.
\end{definition}

\subsection{OSC and GCC imply unique continuation}\label{sec: OSC+GCC>UCP} 
We first prove the following result. 
\begin{proposition}[GCC implies UCP up to $\mathcal{S}^2_{c}$]\label{prop-ucp-S2}
  Let $T>0$ and let $G\subset [0, T]\times [0, 2\pi]$ satisfy \blackhyperref{def:GCC}{$({\rm GCC})$}.
    If $u$ solves the wave equation \eqref{eq: wave-eq-0} and $u_t = 0$ on $G$, then $(\partial_\xi u, \partial_\eta u)|_{t=0}$ belongs to the  class $\mathcal{S}^2_c$ of symmetric function pairs (see Definition \ref{def:classSc}).
\end{proposition}

We start by presenting the following lemma. 
\begin{lemma}\label{lem-gcc-symclass}
Let $T>0$ and let $G\subset [0, T]\times [0, 2\pi]$ satisfy \blackhyperref{def:GCC}{$({\rm GCC})$}. Assume that $f,g\in L^1(\T)$ such that $\int_{\T}(f+g) \d x=0$ and
\begin{align}\label{equ-260106-1}
f(x+t)= g(x-t) \quad \mbox{ for almost every  } (t,x)\in G.
\end{align}
Then $(f,g)$ belongs to the  class $\mathcal{S}^2_c$ of symmetric function pairs.
\end{lemma}
\begin{proof}
In this proof, we identify $(t,x+2k\pi)$ with $(t,x)$ for all $k\in\Z$.
By the assumption \eqref{equ-260106-1}, there exists a set $\mathcal{N} \subset [0,T] \times [0, 2\pi)$ such that $\mbox{meas}_{\R^2}(\mathcal{N})=0$ and
\begin{align}\label{equ-260106-1.5}
f(x+t)=g(x-t) \quad \mbox{ for all } (t,x)\in G\setminus \mathcal{N}.
\end{align}

Because $\mathcal{N}$ has zero measure in $\R^2$, by Fubini's theorem, one has $\int_0^T 1_{\mathcal{N}}(s, x+s) \d s = 0$ for almost every $x\in [0, 2\pi]$.
By the \blackhyperref{def:GCC}{$({\rm GCC})$} assumption, for almost every $x \in [0, 2\pi]$, the sets
\[
I_x = \{ s \in [0,T] : (s, x + s) \in G \setminus \mathcal{N} \}, \quad J_x = \{ s \in [0,T] : (s, x - s) \in G \setminus \mathcal{N} \}
\]
satisfy $\mbox{meas}_{\R}(I_x) \geq c_0/\sqrt{2}$ and $\mbox{meas}_{\R}(J_x) \geq c_0/\sqrt{2}$. Denote this set for $x$ as $\mathcal{M}$.
\vspace{2mm}

For any $x\in \mathcal{M}$ and for any $s \in I_x$, we have $(s, x + s) \in G \setminus \mathcal{N}$, thus
\[
g(x) = f(x + 2s \bmod 2\pi),
\]
because $g(x)=g((x+s)-s)=f((x+s)+s)$ by \eqref{equ-260106-1.5}. Similarly, for  every $s \in J_x$, we have $f(x) = g(x - 2s \bmod 2\pi)$. Define  $E_x, F_x\subset [0, 2\pi]$ as
\begin{align*}
    E_x &= \{ x + 2s \bmod 2\pi : s \in I_x \} \subseteq f^{-1}(g(x)), \\  F_x &= \{ x - 2s \bmod 2\pi : s \in J_x \} \subseteq g^{-1}(f(x)).
\end{align*}
Here, and throughout this proof, the preimage $f^{-1}(z)$ is understood as a subset of $[0,2\pi)$.
For any $y \in [0, 2\pi)$, the equation $x+ 2s \equiv y \pmod{2\pi}$ has solutions $s = (y-x)/2 + k\pi$ for integers $k$. Therefore, for every element $y$ from $E_x$,  
the number of  $s\in I_x\subset [0,T]$ such that $x+ 2s \equiv y \pmod{2\pi}$ is at most $M := \lfloor T/\pi \rfloor + 2$. It follows that
\[
\mbox{meas}_{\R}(I_x)  \leq  M \, \mbox{meas}_{\R}(E_x).
\]
Thus, $\mbox{meas}_{\R}(E_x) \geq \frac{1}{M} \mbox{meas}_{\R}(I_x) \geq \frac{c_0}{\sqrt{2} M}$. Similarly, $\mbox{meas}_{\R}(F_x) \geq \frac{c_0}{\sqrt{2} M}$. Hence,  
\begin{equation}\label{equ-260112}
\begin{split}
    \mbox{meas}_{\R}(f^{-1}(g(x)))\geq \mbox{meas}_{\R}(E_x) \geq \frac{c_0}{\sqrt{2} M},  \\ \mbox{meas}_{\R}(g^{-1}(f(x)))\geq \mbox{meas}_{\R}(F_x) \geq \frac{c_0}{\sqrt{2}M},
    \end{split}
\end{equation} 
for  every $x\in \mathcal{M}$.
\vspace{2mm}

Define the sets
\begin{align*}
    S &= \left\{ a \in \R : \mbox{meas}_{\R}(f^{-1}(a)) \geq \frac{c_0}{\sqrt{2}M} \right\},  
    \\
     S' &= \left\{ b \in \R : \mbox{meas}_{\R}(g^{-1}(b)) \geq \frac{c_0}{\sqrt{2}M} \right\}.
\end{align*}
As illustrated in \eqref{equ-260112}, for almost every $x\in [0, 2\pi]$, we have $g(x) \in S$ and $f(x) \in S'$. 
\vspace{1mm}

We further claim that
$$
S=S'.
$$
Indeed, if $a\in S$, then $\mbox{meas}_{\R}(f^{-1}(a)) \geq \frac{c_0}{\sqrt{2} M}$. 
Since the set $Y = \{ y \in [0, 2\pi] : \mbox{meas}_\R(g^{-1}(f(y)))\geq \frac{c_0}{\sqrt{2}M}\}$ has full measure (i.e. $\mbox{meas}_\R (Y)=2\pi$), the intersection $f^{-1}(a)\cap Y$ has the same measure as $f^{-1}(a)$, hence larger than $\frac{c_0}{\sqrt{2} M}$. For each $y\in f^{-1}(a)\cap Y$, we have $\mbox{meas}_\R(g^{-1}(f(y)))\geq \frac{c_0}{\sqrt{2}M}$. But $f(y)=a$, so $\mbox{meas}_\R(g^{-1}(a))\geq \frac{c_0}{\sqrt{2}M}$. Therefore, $a\in S'$. This proves $S\subset S'$. Similarly, we also have $S'\subset S$. Thus the claim follows. 
\vspace{1mm}

Next, we show that $S$ is a finite set.  For any distinct $a, a' \in S$, the sets $f^{-1}(a)$ and $f^{-1}(a')$ are disjoint. Each such set has measure at least $\frac{c_0}{\sqrt{2} M}$, and since $\mbox{meas}_{\R}([0, 2\pi]) = 2\pi$, we have
\[
\# S \cdot \frac{c_0}{\sqrt{2}M} \leq \sum_{a \in S} \mbox{meas}_\R(f^{-1}(a)) \leq 2\pi.
\]

Without loss of generality, we write $S = \{s_1, \dots, s_K\}$ with the real numbers $s_k$ pairwise distinct.
Now set $A_k = f^{-1}(s_k)$ and $B_k = g^{-1}(s_k)$. Clearly, 
$$\mbox{meas}_{\R}(A_k),  \quad \mbox{meas}_{\R}(B_k)\geq \frac{c_0}{\sqrt{2}M}.$$   
Recall that for almost every $x\in [0, 2\pi]$, we have $f(x), g(x) \in S$.   Thus $(A_k, B_k)_{1\leq k\leq K}$ is a decomposition pair (see Definition \ref{def:decom:pair}), and
\[
f = \sum_{k=1}^K s_k\, \chi_{A_k}, \quad g = \sum_{k=1}^K s_k \,\chi_{B_k}\quad \mbox{  a.e.  in } [0, 2\pi]. 
\]
Then $(f,g)$ is a symmetric function pair (see Definition \ref{def:classSc}), i.e., $(f,g) \in \mathcal{S}^2_c$.
\end{proof}

Now, we come back to the proof of Proposition \ref{prop-ucp-S2}.
\begin{proof}[Proof of Proposition \ref{prop-ucp-S2}]
Because $u$ is a solution to the wave equation \eqref{eq: wave-eq-0}, the functions $\partial_\xi u, \partial_\eta u$ solve the transport equations
$$
(\partial_t-\partial_x)\partial_\xi u=0, \quad (\partial_t+\partial_x)\partial_\eta u=0
$$
with initial conditions
$$
(\partial_\xi u, \partial_\eta u)|_{t=0}=\left(\frac{u_{0x}+u_1}{2}, \frac{u_{0x}-u_1}{2}\right):=(f, g)\in L^2(\T)\times L^2(\T),
$$
satisfying $\int_{\T}(f+g) \d x=0$.

Thanks to Lemma \ref{lem-char} (or directly from Lemma \ref{lem-tran-so}), we have
\begin{align}\label{equ-0107-1}
\partial_\xi u(t,x)=f(x+t), \quad \partial_\eta u(t,x)=g(x-t), \quad \mbox{ a.e. } (t,x)\in [0,T]\times \T.   
\end{align}

Since $\partial_tu=\partial_{\xi}u-\partial_{\eta}u$, the assumption $\partial_t u=0$ almost everywhere in $G$ implies that 
\begin{equation}\label{eq: vanish-G}
\partial_{\xi}u=\partial_{\eta}u, \;   \text{ a.e. in }G.
\end{equation}
It follows from \eqref{equ-0107-1} and \eqref{eq: vanish-G} that
$$
f(x+t)=g(x-t) \quad \mbox{ for \;  a.e. } (t,x)\in G\subset [0, T]\times [0, 2\pi].
$$
By Lemma \ref{lem-gcc-symclass}, we find that $(f,g)$, and thus $(\partial_\xi u, \partial_\eta u)|_{t=0}$, belongs to $\mathcal{S}^2_c$. 
\end{proof}

Finally, as a consequence of \blackhyperref{def:OSC}{$({\rm OSC})$} and Proposition \ref{prop-ucp-S2}, we obtain:
\begin{proposition}[Unique continuation]\label{cor-GCC-sym}
Let $T>0$. Let $G\subset [0, T]\times [0, 2\pi]$ satisfies \blackhyperref{def:GCC}{$({\rm GCC})$} and  \blackhyperref{def:OSC}{$({\rm OSC})$}.
    If $u$ solves the wave equation \eqref{eq: wave-eq-0} and $\partial_t u = 0$ in $G$,  then $u$ is a constant.
\end{proposition}
\begin{proof}
 Thanks to Proposition \ref{prop-ucp-S2}, we know that $(\partial_{\xi}u|_{t= 0}, \partial_{\eta}u|_{t=0})\in \mathcal{S}^2_c$. Assume 
\begin{equation}\label{equ-916-3}
    \partial_{\xi}u|_{t= 0}= \sum_{k= 1}^K s_k \,\chi_{A_k} , \quad \mbox{ a.e. }  x\in \T
\end{equation}
and
\begin{equation}\label{equ-916-4}
    \partial_{\eta}u|_{t=0}= \sum_{k= 1}^K s_k \, \chi_{B_k}, \quad \mbox{ a.e. }  x\in \T
\end{equation}
where $\{(A_k, B_k\}_{k= 1}^K$ is a decomposition pair, $\{s_k\}_{k= 1}^K$ are pairwise distinct real numbers.
We split the discussion into two cases.
\vspace{1mm}

{\it Case (1).} Assume $K=1$. In this case, by \eqref{equ-916-3} and \eqref{equ-916-4} we have $\partial_{\xi}u|_{t=0}=\partial_{\eta}u|_{t=0}=s_1$ for almost every $x\in \T$. 
Since $\int_{\T} (\partial_{\xi}u|_{t=0} + \partial_{\eta}u|_{t=0}) \d x= 0$, thus $s_1= 0$. 
Hence
\begin{align*}
\partial_tu(0,x)=0,\;\partial_xu(0,x)=0 \quad \mbox{ a.e. } x\in \T.
\end{align*}
This implies that the solution $u$ is a constant.
\vspace{2mm}

{\it Case (2).} Assume $K\geq 2$.
 In the sequel, we shall show that this case will never happen, due to \blackhyperref{def:OSC}{$({\rm OSC})$}. In fact, thanks to \eqref{eq: vanish-G}, we know that up to a set with zero measure,
\begin{align}\label{equ-0107-2}
G\subset \{(t,x)\in [0,T]\times [0, 2\pi]:\partial_{\xi}u=\partial_{\eta}u \}.   
\end{align}  

Thanks to \eqref{equ-916-3}-\eqref{equ-916-4}, using Lemma \ref{lem-char}, we have for all $k=1,2,\cdots,K$,
$$
\partial_\xi u= s_k  \quad \mbox{ a.e. in } L_{\xi\in A_k}, 
$$
and
$$
\partial_\eta u= s_k  \quad \mbox{ a.e. in } L_{\eta\in B_k}.
$$
Thus, up to a set with zero measure,
$$
\{(t,x)\in [0,T]\times [0, 2\pi]:\partial_{\xi}u=\partial_{\eta}u  \} \subset \bigcup_{k=1}^K \Big(L_{\xi\in A_k}\cap L_{\eta\in B_k}\cap \left([0,T]\times [0, 2\pi]\right)\Big).
$$
Therefore,  up to a set with zero measure,
\begin{align*}
    G\subset \bigcup_{k=1}^K \Big(L_{\xi\in A_k}\cap L_{\eta\in B_k}\cap \left([0,T]\times [0, 2\pi]\right)\Big).
    \end{align*} 

For each $k\in \{1,\ldots,K\}$,  define $G_k:=G\cap \left(L_{\xi\in A_k}\cap L_{\eta\in B_k}\right)$.
Thus 
\begin{align}\label{equ-916-5}
G = \bigcup_{k=1}^K G_k \quad  \textrm{ and } \quad \{G_k\}_{k= 1}^K \textrm{ are pairwise disjoint sets. }
\end{align}  
By the definition of $G_k$, for each $k$, up to a set of zero measure, 
\begin{equation*}
     G\cap L_{\xi\in A_k}= G_k= G\cap  L_{\eta\in B_k}.
\end{equation*}
This implies that
$$
\mbox{meas}_{\R^2} \Big( [G\cap L_{\xi\in A_{k}}] \, \Delta \, [G\cap L_{\eta\in B_{k}}]\Big)=0.
$$

Since $K\geq 2$, for each $k\in \{1,\ldots,K\}$,   the pair $(A_{k}, B_{k})$ is non-tirival.
Thus $G$ is $(A_{k}, B_{k}; T)$-observable symmetric. 
This is in contradiction to the  \blackhyperref{def:OSC}{$({\rm OSC})$} assumption.  

This shows that case (2) cannot occur.
Thus, the proof is complete.
\end{proof} 
\vspace{1mm}

\subsection{OSC and weak GCC  imply unique continuation}\label{sec: OSC+WGCC>UCP}
In this section, we generalize the results in the previous section.
\begin{proposition}[Weak GCC implies UCP up to $\mathcal{S}^2$]\label{prop-weakGCC-ucp-S2}
    Let $T>0$ and let $G\subset [0, T]\times [0, 2\pi]$ satisfy \blackhyperref{def:weakGCC}{{\rm (weak GCC)}}.
    If $u$ solves the wave equation \eqref{eq: wave-eq-0} and $u_t = 0$ on $G$, then  $(\partial_\xi u, \partial_\eta u)|_{t=0}$ belongs to  the class $\mathcal{S}^2$ of weak symmetric function pairs (see Definition \ref{def:classS}).
\end{proposition}

Similar to Proposition \ref{prop-ucp-S2}, Proposition \ref{prop-weakGCC-ucp-S2} follows from an auxiliary lemma.
\begin{lemma}\label{lem-gcc-symclass-weak}
Let $T>0$ and let $G\subset [0, T]\times [0, 2\pi]$ satisfy the \blackhyperref{def:weakGCC}{{\rm (weak GCC)}}. Assume that $f,g\in L^2(\T)$ such that $\int_\T (f+g) \d x= 0$ and
\begin{align}\label{equ-260106-2}
f(x+t)=g(x-t) \quad \mbox{ for  almost every } (t,x)\in G.
\end{align}
Then $(f,g)$ belongs to the class $\mathcal{S}^2$ of weak symmetric function pairs.
\end{lemma}

\begin{proof}
The idea is similar to the proof of Lemma \ref{lem-gcc-symclass}, and we use the same notations. Since $G$ satisfies \blackhyperref{def:weakGCC}{{\rm (weak GCC)}},  for almost every $x \in \T$, the sets
\[
I_x = \{ s \in [0,T] : (s, x + s) \in G \setminus \mathcal{N} \}, \quad J_x = \{ s \in [0,T] : (s, x - s) \in G \setminus \mathcal{N} \}
\]
are measurable and satisfy $\mbox{meas}_{\R}(I_x) > 0$ and $\mbox{meas}_{\R}(J_x) > 0$.   
For such $x\in \mathcal{M}$, define
\[
E_x = \{ x + 2s \bmod 2\pi : s \in I_x \} \subseteq f^{-1}(g(x)), \quad F_x = \{ x - 2s \bmod 2\pi : s \in J_x \} \subseteq g^{-1}(f(x)).
\]
Then 
\begin{equation*}
\mbox{meas}_{\R}(E_x) \geq  M \mbox{meas}_{\R}(I_x) > 0, \quad   \mbox{meas}_{\R}(F_x) \geq M \mbox{meas}_{\R}(J_x) > 0.  
\end{equation*}
Consequently,
\[
\mbox{meas}_{\R}(f^{-1}(g(x))) > 0, \quad \mbox{meas}_{\R}(g^{-1}(f(x))) > 0.
\]

Define the sets
\[
S = \{ a \in \R : \mbox{meas}_{\R}(f^{-1}(a)) > 0 \}, \quad S' = \{ b \in \R : \mbox{meas}_{\R}(g^{-1}(b)) > 0 \}.
\]
From the above, for almost every $x$, we have $g(x) \in S$ and $f(x) \in S'$. If $a \in S$, then $\mbox{meas}_{\R}(f^{-1}(a)) > 0$. And $f^{-1}(a)\cap\mathcal{M}$ has positive measure. For  every $y \in f^{-1}(a)\cap\mathcal{M}$, we have $\mbox{meas}_{\R}(g^{-1}(f(y)) > 0$. But $f(y)= a$, so  $\mbox{meas}_{\R}(g^{-1}(a)) > 0$, i.e., $a \in S'$. Symmetrically, $S' \subset S$. Hence, $S = S'$.
\vspace{1mm}

For distinct $a, a' \in S$, the sets $f^{-1}(a)$ and $f^{-1}(a')$ are disjoint. Since each has positive measure and $\mbox{meas}_{\R}([0, 2\pi]) < \infty$, there can be at most countably many such $a$. Therefore, $S$ is at most countable. So we assume that
$$
S=\{s_1,s_2,\cdots,s_k,\cdots\}
$$
where $s_k$ are pairwise distinct real numbers. For each $k$, define $A_k=f^{-1}(s_k)$ and $B_k=g^{-1}(s_k)$. Then $(A_k, B_k)_{1\leq k\leq K}$ is a weak decomposition pair (see Definition \ref{def:weakdecpair}). Since $f\in L^2(\T)$, we have
$$
\|f\|^2_{L^2(\T)}=\sum_{k\in S}|s_k|^2\mbox{meas}_{\R}(A_k)<\infty.
$$
This implies that
$$
f=\sum_{k\in S} s_k \, \chi_{A_k} \quad \mbox{ in } L^2(\T).
$$
Similarly, 
$$
g=\sum_{k\in S} s_k \, \chi_{B_k} \quad \mbox{ in } L^2(\T).
$$
Thus $(f,g)$ is a weak symmetric function pair.
\end{proof}

\begin{proposition}[OSC and weak GCC  imply UCP]\label{coro:osc+weakGCCtoUCP}
Let $T>0$ and $G\subset [0,T]\times [0, 2\pi]$ be a measurable set. Let $G$ satisfies  \blackhyperref{def:weakGCC}{{\rm (weak GCC)}} and  \blackhyperref{def:OSC}{$({\rm OSC})$}.
    If $u$ solves the wave equation \eqref{eq: wave-eq-0} and $u_t = 0$ on $G$,  then $u$ is a constant.
\end{proposition}
\begin{proof}
With Proposition \ref{prop-weakGCC-ucp-S2}, the result follows by an argument analogous to the proof of Proposition \ref{cor-GCC-sym}. Indeed, the \blackhyperref{def:weakGCC}{{\rm (GCC)}} condition guarantees that the sets $A_k$ and $B_k$ in \eqref{equ-916-3} and \eqref{equ-916-4} satisfy a uniform lower bound.
However, this uniform bound is not actually used in the proof of Proposition \ref{cor-GCC-sym}, particularly in case (2). We only require that $A_k$ and $B_k$ have positive measure, which is already ensured by the \blackhyperref{def:weakGCC}{{\rm (weak GCC)}} condition via Proposition \ref{prop-weakGCC-ucp-S2}.  To obtain a contradiction, we show that $G$ is $(A_k, B_k; T)$-symmetric observable.
\end{proof}

\subsection{Proof of Theorem \ref{thm: main-ucp}} Collect Propositions \ref{prop:UCtoOSC}, \ref{prop:ucp-weakGCC}, and \ref{coro:osc+weakGCCtoUCP}.

\section{Observability}\label{sec:sharp:obser}
This section is devoted to the proof of Theorem \ref{thm: main-ob} concerning observability.
\vspace{1mm}
\subsection{Transport equations}
As a preparation, we first investigate the observability for the transport equations:
\begin{align}
(\partial_t-\partial_x)w=0,\quad w|_{t=0}= w_0\in L^2(\T),\label{eq: transport-eq}\\
(\partial_t+\partial_x) w=0,\quad w|_{t=0}= w_0\in L^2(\T).\label{eq: transport-eq-2}
\end{align}
Let $T>0$ and $G\subset [0,T]\times \T$ be a measurable set. The equation \eqref{eq: transport-eq} is said to be {\it observable on $G$} if there exists a constant $C>0$ such that 
\begin{align}\label{equ-tran-suff-necc-1}
 \|w_0\|^2_{L^2(\T)}\leq C\iint_G|w(t,x)|^2\d x \d t 
\end{align}
for all solutions to the transport equation \eqref{eq: transport-eq}. 

We obtain the following result, and put the proof in Appendix \ref{proof-lem-char}.
\begin{proposition}\label{thm-tran-suff-nece}
Let $T>0$ and $G\subset [0,T]\times [0, 2\pi]$ be a measurable set.
The tranport equation  \eqref{eq: transport-eq} is observable on $G$ if and only if there exists a constant $c_0>0$ such that 
\begin{equation*}
      \textrm{ for } a.e. \; \; x\in [0, 2\pi],  \; \mbox{meas}_\R(G\cap L_{\xi= x})\geq c_0;
\end{equation*} 
The equation  \eqref{eq: transport-eq-2} is observable on $G$ if and only if there exists a constant $c_0>0$ such that 
\begin{equation*}
      \textrm{ for } a.e. \; \; x\in [0, 2\pi],  \; \mbox{meas}_\R(G\cap L_{\eta= x})\geq c_0.
\end{equation*} 
\end{proposition}

\subsection{Observability implies GCC}\label{sec:obimplyGCC}
Now we consider wave equations, and show the following result; its proof is left in Appendix \ref{sec:app:3}.
\begin{proposition}\label{wave-GCC-nece}
Let $T>0$ and $G\subset [0,T]\times [0, 2\pi]$ be a measurable set. Consider the wave equation \eqref{eq: wave-eq-0}.  If the observability \eqref{eq: wave-ob} holds on the set $G$, then $G$ satisfies \blackhyperref{def:GCC}{{\rm (GCC)}}.
\end{proposition}

\subsection{GCC implies weak observability}\label{sec: weak-ob-sec}
Recall the spaces  of zero-mean functions/distributions, $\dot L^2(\T)\subset L^2(\T), \dot H^1(\T)\subset H^1(\T)$, 
and $\dot H^{-1}(\T)=(\dot H^1(\T))'$. In the classical cylindrical setting $[0,T]\times\omega$, GCC implies a weak observability. We establish such a standard property in the spacetime measurable setting.
\begin{proposition}[Weak observability]\label{prop: weak-wave-ob}
Let $T>0$. Let a measurable set $G\subset[0,T]\times[0, 2\pi]$  satisfying \blackhyperref{def:GCC}{$({\rm GCC})$}. Then there exists a constant $C>0$ such that
\begin{equation}\label{eq: weak-ob-wave}
\|\partial_x u_0\|^2_{ L^2(\T)}+\|u_1\|^2_{L^2(\T)}\leq C\left(\iint_{G}|\partial_tu(t,x)|^2\d x\d t+\|u_0\|^2_{\dot L^2(\T)}+\|u_1\|^2_{H^{-1}(\T)}\right),
\end{equation}
holds for every solution $u$ to the wave equation \eqref{eq: wave-eq-0} with $(u_0, u_1)\in \dot H^1(\T)\times L^2(\T)$.
\end{proposition}

To prove this result, we first derive a high-frequency estimate, which is based on the observability for transport equations; then, we obtain the weak observability using the standard argument.  The detailled  proof is left in Appendix \ref{sec:app:weak-obser-wave}.

\subsection{OSC and GCC imply observability}\label{sec:OSCplusGCCimplOBS}
Combining the unique continuation property, Proposition \ref{cor-GCC-sym},  and the weak observability result, Proposition \ref{prop: weak-wave-ob}, the observability follows from the standard compactness--uniqueness argument introduced by Bardos--Lebeau--Rauch~\cite{BLR-gcc}.
\begin{proposition}\label{prop:osc+GCCtoObs}
Let $T>0$ and $G\subset [0,T]\times [0, 2\pi]$ be a measurable set. If $G$ satisfies  \blackhyperref{def:GCC}{{\rm (GCC)}} and  \blackhyperref{def:OSC}{$({\rm OSC})$},
then the observability \eqref{eq: wave-ob} holds on the set $G$.
\end{proposition}

We argue by contradiction.
Suppose that \eqref{eq: wave-ob} fails. Then, by modulating the zero mode, there exists a sequence of initial data  
$\{(u_{0,n}, u_{1,n}) \in \dot H^1(\mathbb{T}) \times L^2(\mathbb{T})\}_{n\in \N}$, such that  
\[
\|\partial_x u_{0,n}\|_{L^2}^2 + \|u_{1,n}\|_{L^2}^2 = 1,
\]  
and  
\begin{equation}\label{equ-proof-suff-1}
\iint_{G} |\partial_t u_n(t,x)|^2   \d x \d t \to 0 \quad \text{as } n \to \infty,
\end{equation}  
where $u_n$ is the solution of \eqref{eq: wave-eq-0} with initial data $(u_{0,n},u_{1,n})$. By the Banach--Alaoglu theorem, we can extract a subsequence, still denoted by $\{(u_{0,n},u_{1,n})\}_{n\in \mathbb{N}}$, that converges weakly in $\dot H^1(\mathbb{T}) \times L^2(\mathbb{T})$ to some limit $(u_0,u_1)$. 
Moreover, the solutions \{$u_n\}_{n\in \mathbb{N}}$ satisfy that, for each $t\in[0,T]$,
\[
(u_n(t),\partial_t u_n(t))\rightharpoonup (u(t),\partial_t u(t))
\quad \textrm{weakly in } H^1(\T)\times L^2(\T),
\]
where $u$ is the solution of \eqref{eq: wave-eq-0}
with initial data $(u_0,u_1)\in \dot H^1(\mathbb{T}) \times L^2(\mathbb{T})$.

On the one hand, from \eqref{equ-proof-suff-1} we obtain for the limit  
\[
\iint_{G} |\partial_t u|^2\d x \d t = 0,
\]  
which implies $\partial_t u = 0$ almost everywhere in $G$. Then, the unique continuation property,  Proposition \ref{cor-GCC-sym},  yields that $u$ is a constant. This implies that $u_0= u_1= 0$.

On the other hand, since $u_n$ solves \eqref{eq: wave-eq-0}, Proposition \ref{prop: weak-wave-ob} gives the weak observability estimate  
\begin{equation}\label{equ-proof-suff-2}
1 = \|\partial_x u_{0,n}\|_{L^2}^2 + \|u_{1,n}\|_{L^2}^2
\leq C\Bigl( \iint_{G} |\partial_t u_n(t,x)|^2 \d x \d t
+ \|u_{0,n}\|_{\dot L^2(\mathbb{T})}^2 + \|u_{1,n}\|_{H^{-1}(\mathbb{T})}^2 \Bigr).
\end{equation}
Because the embedding $\dot H^1 \times L^2 \hookrightarrow \dot L^2 \times  H^{-1}$ is compact, the convergence $(u_{0,n}, u_{1,n}) \to (u_0,u_1)$ holds strongly in $\dot L^2(\mathbb{T}) \times  H^{-1}(\mathbb{T})$. Passing to the limit $n \to \infty$ in \eqref{equ-proof-suff-2} yields  
\[
1 \leq C\bigl( \|u_{0}\|_{\dot L^2(\mathbb{T})}^2 + \|u_{1}\|_{H^{-1}(\mathbb{T})}^2 \bigr),
\]  
which contradicts the fact that  $u_0= u_1= 0$. Consequently, the inequality \eqref{eq: wave-ob} must hold.

\subsection{Proof of Theorem \ref{thm: main-ob}} Collect Propositions \ref{prop:UCtoOSC}, \ref{wave-GCC-nece}, and \ref{prop:osc+GCCtoObs}.
\vspace{2mm}

\appendix

\section{Proof of Lemma \ref{lem-Green-3}} 
\label{sec:app:integral}

For any $A, B$ subsets of $[0,2\pi]$,  define 
\begin{align*}
    \Omega^+(B; T) &:= L_{\eta\in B}\cap \left([0, T]\times [0, 2\pi]\right), \\
    \Omega^-(A; T) &:= L_{\xi\in A}\cap \left([0, T]\times [0, 2\pi]\right).
\end{align*}

\begin{lemma}\label{lem-Green-1}
Let $T>0$. Let $A, B \subset [0, 2\pi]$ be measurable sets.
For any function $w \in C^{1}(\mathbb{R}^{2})$ that is $2\pi$-periodic with respect to the $x$ variable, one has
\begin{align*}
\iint_{\Omega^+(B; T)} \bigl( \partial_x w + \partial_t w \bigr) \d x\d t
&= \int_{B+ T} w(x,T)\d x - \int_{B} w(x,0)\d x, \\
\iint_{\Omega^-(A; T)} \bigl( \partial_x w - \partial_t w \bigr)\d x\d t
&= \int_{A} w(x,0)\d x - \int_{A-T} w(x,T)\d x,
\end{align*}
 where $A-T = \{x-T \textrm{ mod } 2\pi: x \in A\}, B+ T = \{x+ T \textrm{ mod } 2\pi : x \in B\}$
\end{lemma}
For the case that $A, B$ are open intervals, this lemma is a direct consequence of Green's formula and the $2\pi$-periodicity of the function. As for the case that $A, B$ are measurable sets, one can approximate $A$ and $B$ by the unions of finitely many open intervals, and obtain the result by passing the limit. 

\vspace{2mm}
Now, we proceed with the proof of Lemma \ref{lem-Green-3}.
By a standard density argument, it suffices to consider solutions that are regular enough, namely, $G$ is the union of finitely many open cubes,  $f \, \chi_{G}\in C([0, T]\times [0, 2\pi])$, and $u\in C^2([0, T]\times [0, 2\pi])$. 

Applying Lemma \ref{lem-Green-1} to $w= U$, we obtain
\[
\iint_{\Omega^-(A; T)} \bigl( \partial_x U - \partial_t U \bigr) \d x \d t
= \int_{A} U(x,0)\d x - \int_{A-T} U(x,T) \d x.  \] 
By the definition of $U$, using the wave equation, we have
$$
\partial_x U - \partial_t U= (\partial_x^2-\partial_t^2)u= -f \chi_G \quad \mbox{ in } [0, T]\times [0, 2\pi].
$$
Thus
$$
\int_{A} U(x,0)\d x - \int_{A-T} U(x,T)\d x=
-\iint_{\Omega^-(A; T)}f \chi_G \d x \d t=-\iint_{G\cap L_{\xi\in A}^T} f \d x \d t.
$$
This proves the first equality in Lemma \ref{lem-Green-3}. 
Similarly, we obtain the second equality. 
\vspace{1mm}

\section{Some results on transport equations}\label{proof-lem-char}
\subsection{Proof of Lemma \ref{lem-char}}
The following result is classical; it can be proved either by Fourier series
or by the method of characteristics combined with a density argument.
\begin{lemma}\label{lem-tran-so}
Let $h\in L^2(\T)$. Then the Cauchy problem
\[
\partial_t w \pm \partial_x w = 0, \qquad w(0,\cdot)=h(\cdot),
\]
admits a unique solution $w\in C(\R;L^2(\T))$, which is explicitly given by
\[
w(t,\cdot)=h(\cdot\mp t)\quad \text{in } L^2(\T), \qquad \forall\, t\in\R.
\]
In particular,
\[
w(t,x)=h(x\mp t)\quad \text{for a.e. }(t,x)\in\R\times\T.
\]
\end{lemma}

\begin{proof}[Proof of Lemma \ref{lem-char}]
We only prove the result for $w:=\partial_\eta u$. According to the wave equation, we know $\partial_\xi w=0$, or equivalently 
$$
\partial_t w+\partial_x w=0, \quad w(0, \cdot)= g(\cdot)\in L^2(\T)
$$
where $g= (- \partial_t+ \partial_x)u|_{t=0}$. Thanks to Lemma \ref{lem-tran-so}, 
$$
w(t,x)=g(x-t) \quad \mbox{ for a.e. } (t,x)\in \R\times \T.
$$
Thus, under the null coordinate
$$
\partial_{\eta}u (\xi, \eta)= g(\eta) \quad \mbox{ for a.e. } (\xi, \eta)\in \R\times \R.
$$
Then Fubini's theorem implies the result on $\eta$-characteristics.
\end{proof}

\subsection{Proof of Proposition \ref{thm-tran-suff-nece}} Now, we demonstrate this result concerning the necessary and sufficient condition for the observability of transport equations.
We only prove the observability for \eqref{eq: transport-eq}. 
Let $\widehat{w}(k)$ be the $k$-th Fourier coefficient of $w$ defined as
$
\widehat{w}(k)=\frac{1}{2 \pi} \int_0^{2 \pi} w(x) \ee^{-\ii k x} \d x.
$ 
We write the initial state $w_0$ and its associated solution $w$ to \eqref{eq: transport-eq} into Fourier series
\[
w_0(x)=\sum_{k\in\Z}\widehat{w_0}(k)\ee^{\ii kx},\; w(t,x)=\sum_{k\in\Z}\widehat{w_0}(k)\ee^{\ii k(t+x)}.
\]
Then the observability inequality \eqref{equ-tran-suff-necc-1}
is equivalent to
\[
\sum_{k\in\Z}|\widehat{w_0}(k)|^2\leq C\iint_G|\sum_{k\in\Z}\widehat{w_0}(k)\ee^{\ii k(t+x)}|^2\d x\d t.
\]

\begin{lemma}\label{prop: tran-eq-ob}
Let $T>0$ and $G\subset [0,T]\times [0, 2\pi]$ satisfy \blackhyperref{def:GCC}{{\rm (GCC)}}. Then for any   $\{a_k\}_{k\in\Z}\in l^2(\mathbb{C})$, we have
\begin{align}
\sqrt{2}\pi c_0\sum_{k\in\Z}|a_k|^2&\leq \iint_G|\sum_{k\in\Z}a_k\ee^{\ii k(t+x)}|^2\d x\d t,\label{eq: ob-transport}\\
\sqrt{2}\pi c_0\sum_{k\in\Z}|a_k|^2&\leq \iint_G|\sum_{k\in\Z}a_k\ee^{-\ii k(t-x)}|^2\d x\d t.\label{eq: ob-transport-2}
\end{align}
\end{lemma}
\begin{proof}
In this proof, we identify $(t,x+2k\pi)$ with $(t,x)$ for all $k\in\Z$.
By Fubini's theorem, we have
\begin{align*}
\iint_G|\sum_{k\in\Z}a_k\ee^{\ii k(t+x)}|^2\d x\d t&=\int_{\T}\int_0^T|\sum_{k\in\Z}a_k\ee^{\ii k(t+x)}|^2\chi_G(t,x)\d t\d x\\
&=\int_{\T}|\sum_{k\in\Z}a_k\ee^{\ii ky}|^2\int_0^T\chi_G(t,y-t)\d t\d y.
\end{align*}
Using \blackhyperref{def:GCC}{({\rm GCC})}, we know that for a.e. $y\in\T$, $\sqrt{2}\int_0^T \chi_G(t,y-t)\d t\geq c_0$. Therefore, 
\[
\sqrt{2} |\sum_{k\in\Z}a_k\ee^{\ii ky}|^2\int_0^T \chi_G(t,y-t)\d t\geq c_0|\sum_{k\in\Z}a_k\ee^{\ii ky}|^2,\; \mbox{ for a.e. } y\in\T.
\]
By Plancherel's theorem, $\int_{\T}|\sum_{k\in\Z}a_k\ee^{\ii ky}|^2\d y=
2\pi\sum_{k\in\Z}|a_k|^2$. Consequently, we obtain the observability \eqref{eq: ob-transport}.
The inequality \eqref{eq: ob-transport-2} follows similarly.
\end{proof}

\begin{proof}[Proof of Proposition \ref{thm-tran-suff-nece}]
We only prove the result corresponding to  \eqref{eq: transport-eq}, since the other equation can be treated in the same way. The ``if'' part  follows directly from Lemma \ref{prop: tran-eq-ob}. We now prove the ``only if'' part. Suppose that for any $\varepsilon>0$, there exists a set $E_\varepsilon\subset \T$ with positive measure $|E_\varepsilon|>0$ such that
\begin{equation}\label{equ-tran-suff-necc-3} 
      \textrm{ for } a.e. \; \; x\in E_{\varepsilon},  \; \mbox{meas}_\R(G\cap L_{\xi= x})< \varepsilon.
\end{equation}
Let $w_{0,  \varepsilon}(x)=\chi_{E_\varepsilon}(x)$. Then $w_{0, \varepsilon}\in L^2(\T)$ and its support is contained in $E_\varepsilon$. Then, similar to the proof in Lemma \ref{prop: tran-eq-ob}, we have
\begin{align*}
\iint_G|w(t,x)|^2\d x \d t =\int_{\T}|w_{0, \varepsilon}(x)|^2\int_0^T\chi_G(t,x-t)\d t\d x\leq 2\pi \varepsilon \|w_{0, \varepsilon}\|^2_{L^2(\T)}.  
\end{align*}
Thus the transport equation is not observable. Hence, the observability of \eqref{eq: transport-eq} implies the existence of $c_0>0$ such that for
$ a.e. \; x\in [0, 2\pi],  \; \mbox{meas}_\R(G\cap L_{\xi= x})\geq c_0$.
\end{proof}
\vspace{2mm}

\section{Proof of Proposition \ref{wave-GCC-nece}}\label{sec:app:3}
We only prove the following result for the $\xi$-direction, while the $\eta$-direction can be proved similarly: there exists  a constant $c_0>0$ such that 
\begin{equation}\label{eq:app:D:1}
      \textrm{ for } a.e. \; \; x\in [0, 2\pi],  \; \mbox{meas}_\R(G\cap L_{\xi= x})\geq c_0.
\end{equation}  

Suppose that the observability inequality \eqref{eq: wave-ob} holds for the wave equation \eqref{eq: wave-eq-0} on $G$, 
\begin{equation}\label{eq:app:D:ob:wave}
\|\partial_x u_{0}\|^2_{L^2(\T)}+\|u_1\|^2_{L^2(\T)}\leq C_1\iint_G|\partial_tu(t,x)|^2\d x\d t,\quad \forall (u_0, u_1)\in  H^1(\T)\times L^2(\T).
\end{equation}
We first show that the transport equation 
\begin{equation}\label{eq:app:C:tran:1}
(\partial_t-\partial_x)w=0,\quad w|_{t=0}= w_0\in  \dot L^2(\T)
\end{equation}
satisfies the following observability inequality on $G$,
\begin{align}\label{eq:app:D:ob:tran}
\|w_0\|^2_{L^2(\T)}\leq  \frac{C_1}{2}\iint_G|w(t,x)|^2\d x \d t, \quad \forall w_0 \in  \dot L^2(\T).
\end{align}
Note that this inequality differs from \eqref{equ-tran-suff-necc-1}: in \eqref{eq:app:D:ob:tran} the initial data belongs to $ \dot L^2(\T)$, whereas in \eqref{equ-tran-suff-necc-1} it belongs to $L^2(\T)$.
\vspace{1mm}

Indeed, for any $w_0\in  \dot L^2(\T)$, we are able to find $(u_0, u_1)\in \dot H^1(\T)\times L^2(\T)$ such that $(\partial_x u_{0}, u_1)= (w_0, w_0)$. It satisfies
\begin{equation*}
  \partial_x u_{0}+ u_1= 2 w_0, \quad  \partial_x u_{0}- u_1= 0.
\end{equation*}
Let $u$ be the unique solution of the wave equation \eqref{eq: wave-eq-0} with initial state $(u_0, u_1)$. Define 
\begin{equation*}
w(t, x):= \partial_{\xi} u(t, x), \; \quad \forall (t, x)\in [0, T]\times \T.
\end{equation*}
Then $w$ is the unique solution of the transport equation \eqref{eq:app:C:tran:1} with initial state $w_0$. Using the observability inequality \eqref{eq:app:D:ob:wave} for $u$, and fact that $w= \partial_{\xi}u= \partial_t u$ in $[0, T]\times \T$,  we obtain \eqref{eq:app:D:ob:tran}.
\vspace{1mm}

Next, we show that inequality \eqref{eq:app:D:ob:tran} implies the existence of $c_0$ such that \eqref{eq:app:D:1} holds. The argument is similar to the ``only if'' part of the proof of Proposition \ref{thm-tran-suff-nece}. It suffices to modify the definition of $w_{0, \varepsilon}$ so that  $w_{0, \varepsilon}\in  \dot L^2(\T)$. This yields  \eqref{eq:app:D:1}. Therefore, $G$  satisfies \blackhyperref{def:GCC}{({\rm GCC})}.
\vspace{1mm}

\section{Proof of Proposition \ref{prop: weak-wave-ob}}\label{sec:app:weak-obser-wave}
We first provide the following result.
\begin{proposition}[High-frequency estimates]\label{prop: high-ob-app}
Let $T>0$ and $G\subset[0,T]\times [0, 2\pi]$  satisfy \blackhyperref{def:GCC}{$({\rm GCC})$}. Then there exist constants $N>0$ and $C>0$ depending only on $T$ and $G$, such that for every $(u_0,u_1)\in H^1(\T)\times L^2(\T)$ with $\supp \widehat{u_0} \subset\{|k|>N\}$ and $\supp \widehat{u_1} \subset\{|k|>N\}$, we have
\begin{equation}\label{eq: high-est}
\|\partial_x u_0\|^2_{L^2(\T)}+\|u_1\|^2_{L^2(\T)}\leq C\iint_{G}|\partial_tu(t,x)|^2\d t\d x,
\end{equation}
where $u$ is the unique solution to \eqref{eq: wave-eq-0} with initial data $(u_0,u_1)$.
\end{proposition}
\begin{proof}
By conservation of energy and the time-translation invariance of the wave equation on $\T$,
it suffices to establish \eqref{eq: high-est} for $T\le 2\pi$.
Indeed, for larger times the interval $[0,T]$ can be decomposed into finitely many
subintervals of length at most $2\pi$.
Moreover, since the wave equation is diagonal in Fourier variables,
the high-frequency localization of the initial data is preserved in time,
so the frequency assumption remains valid on each subinterval.

Suppose $\supp \widehat{u_0} \subset\{|k|>N\}$ and $\supp \widehat{u_1}\subset\{|k|>N\}$. We write
\begin{gather*}
u_0(x)=\sum_{|k|>N}\widehat{u_0}(k)\ee^{\ii kx}, \quad u_1(x)=\sum_{|k|>N}\widehat{u_1}(k)\ee^{\ii kx}.
\end{gather*}
Therefore, the solution to \eqref{eq: wave-eq-0} has the decomposition
\begin{gather*}
u(t,x)=\sum_{|k|>N}\left(\frac{\ii k\widehat{u_0}(k)+\widehat{u_1}(k)}{2\ii k}\ee^{\ii k(t+x)}+\frac{\ii k\widehat{u_0}(k)-\widehat{u_1}(k)}{2\ii k}\ee^{-\ii k(t-x)}\right).
\end{gather*}
Then we can rewrite the   right-hand side of \eqref{eq: high-est} as
\begin{align*}
\|\partial_tu\|_{L^2(G)}^2=\iint_G\Big|\sum_{|k|>N}\Big(\ii k\frac{\ii k\widehat{u_0}(k)+\widehat{u_1}(k)}{2\ii k}\ee^{\ii k(t+x)}-\ii k\frac{\ii k\widehat{u_0}(k)-\widehat{u_1}(k)}{2\ii k}\ee^{-\ii k(t-x)}\Big)\Big|^2\d x\d t\\
=\frac{1}{4}\iint_G\Big|\sum_{|k|>N}\left((\ii k\widehat{u_0}(k)+\widehat{u_1}(k))\ee^{\ii k(t+x)}-(\ii k\widehat{u_0}(k)-\widehat{u_1}(k))\ee^{-\ii k(t-x)}\right)\Big|^2\d x\d t.    
\end{align*}
Expand the square term as  
\begin{align*}
&\iint_G\Big|\sum_{|k|>N}\left((\ii k\widehat{u_0}(k)+\widehat{u_1}(k))\ee^{\ii k(t+x)}-(\ii k\widehat{u_0}(k)-\widehat{u_1}(k))\ee^{-\ii k(t-x)}\right)\Big|^2\d x \d t = I_1+I_2+I_3+I_4,
\end{align*}
where $I_j, j=1,2,3,4,$ are given by
\begin{align*}
I_1&=\iint_{G}\big|\sum_{|k|>N}(\ii k\widehat{u_0}(k)+\widehat{u_1}(k))\ee^{\ii k(t+x)}\big|^2\d x\d t,\\
I_2&=-\iint_{G} \sum_{|k|,|l|>N}(\ii k\widehat{u_0}(k)-\widehat{u_1}(k))\overline{(\ii l\widehat{u_0}(l)+\widehat{u_1}(l))}\ee^{-\ii l(t+x)}\ee^{-\ii k(t-x)}\d x\d t,\\
I_3&=-\iint_{G}\sum_{|k|,|l|>N}\overline{(\ii k\widehat{u_0}(k)-\widehat{u_1}(k))} (\ii l\widehat{u_0}(l)+\widehat{u_1}(l))\ee^{\ii l(t+x)}\ee^{\ii k(t-x)}\d x\d t= \overline{I_2},\\
I_4&=\int_{G}\big|\sum_{|k|>N}(\ii k\widehat{u_0}(k)-\widehat{u_1}(k))\ee^{-\ii k(t-x)}\big|^2\d x\d t.
\end{align*}
We first note that the dominating terms are $I_1$ and $I_4$. Using the estimates \eqref{eq: ob-transport} and \eqref{eq: ob-transport-2},
$$
I_1\geq \sqrt{2}\pi c_0\sum_{|k|>N}\big|\ii k\widehat{u_0}(k)+\widehat{u_1}(k)\big|^2,\;\;I_4\geq \sqrt{2}\pi c_0\sum_{|k|>N}\big|\ii k\widehat{u_0}(k)-\widehat{u_1}(k)\big|^2.
$$
It follows that $I_1+I_4\geq 2\sqrt{2}\pi c_0\sum_{|k|>N}(|k|^2\widehat{u_0}(k)|^2+|\widehat{u_1}(k)|^2)$. 
\vspace{2mm}

Next, we turn to the mixed terms $I_2$ and $I_3$. Recall that 
 $T\leq 2\pi$. In this case, $G$ is a subset of $[0, 2\pi]\times [0, 2\pi]$. 
Thus $I_2$ can be rewritten as
$$
I_2=-\iint_{[0, 2\pi]^2}\chi_G(t,x)\sum_{|k|,|l|>N}(\ii k\widehat{u_0}(k)-\widehat{u_1}(k))\overline{(\ii l\widehat{u_0}(l)+\widehat{u_1}(l))}\ee^{-\ii l(t+x)}\ee^{-\ii k(t-x)}\d x\d t.
$$
From now on, we regard the restriction $\chi_{G}|_{[0, 2\pi]^2}$ as an $L^2(\T^2)$ function, and simply denote it by $\chi_{G}$.
This implies that
\begin{align*}
I_2=-\sum_{|k|,|l|>N}(\ii k\widehat{u_0}(k)-\widehat{u_1}(k))\overline{(\ii l\widehat{u_0}(l)+\widehat{u_1}(l))}\widehat{\chi_G}(l+k,l-k).
\end{align*}
By Cauchy--Schwarz inequality, we obtain 
\begin{align*}
|I_2|\leq &\left(\sum_{|k|,|l|>N}|(\ii k\widehat{u_0}(k)-\widehat{u_1}(k))\overline{(\ii l\widehat{u_0}(l)+\widehat{u_1}(l))}|^2\right)^{\frac{1}{2}}\left(\sum_{|k|,|l|>N}|\widehat{\chi_G}(l+k,l-k)|^2\right)^{\frac{1}{2}}\\
\leq &2\sum_{|k|>N}\left( |k|^2|\widehat{u_0}(k)|^2+|\widehat{u_1}(k)|^2\right) \left(\sum_{|k|,|l|>N}|\widehat{\chi_G}(l+k,l-k)|^2\right)^{\frac{1}{2}}.
\end{align*}
Let $k+l=\alpha_1$, $l-k=\alpha_2$, then
\[\{(\alpha_1,\alpha_2)\in\Z^2:|\alpha_1+\alpha_2|>2N,|\alpha_1-\alpha_2|>2N\}\subset\{(\alpha_1,\alpha_2)\in\Z^2:\alpha_1^2+\alpha_2^2>4 N^2\}.
\]
Using the Plancherel theorem, we know that $\sum_{\alpha\in\Z^2}|\widehat{\chi_G}(\alpha)|^2= \|\chi_G\|_{L^2(\T^2)}^2=\mbox{meas}_{\R^2}(G)$, which is finite. Thus, there exists a constant $N>0$ depending only on $G$  such that
\[
\left(\sum_{|\alpha|^2> 4N^2,\alpha\in\Z^2}|\widehat{\chi_G}(\alpha)|^2\right)^{\frac{1}{2}}\leq \frac{c_0}{20}.
\]
With this choice of $N$, we can bound $I_2$ as
$
|I_2| \leq \frac{c_0}{10}\sum_{|k|>N}( |k|^2|\widehat{u_0}(k)|^2+|\widehat{u_1}(k)|^2).
$
Similar estimate holds for $I_3$. Combining the above bounds together, we infer the high-frequency estimate 
\begin{align}\label{eq: high-est-wave-Fourier} 
&\iint_G|\sum_{|k|>N} \left((\ii k\widehat{u_0}(k)+\widehat{u_1}(k))\ee^{\ii k(t+x)}-(\ii k\widehat{u_0}(k)-\widehat{u_1}(k))\ee^{-\ii k(t-x)}\right)|^2\d x\d t\nonumber \\ &\geq I_1+I_4-|I_2|-|I_3|
\geq (2\sqrt{2}\pi-\frac{1}{5})c_0 \left( |k|^2|\widehat{u_0}(k)|^2+|\widehat{u_1}(k)|^2\right).
\end{align}
By the Plancherel theorem, we compute the left-hand side  of \eqref{eq: high-est} as
\begin{align*}
\|\partial_x u_0\|^2_{L^2(\T)}+\|u_1\|^2_{L^2(\T)}=\sum_{|k|>N}\left(|k|^2|\widehat{u_0}(k)|^2+|\widehat{u_1}(k)|^2\right).
\end{align*}
This, together with \eqref{eq: high-est-wave-Fourier}, gives \eqref{eq: high-est}.
\end{proof}

\begin{proof}[Proof of Proposition \ref{prop: weak-wave-ob}]
Based on the high-frequency estimates in Proposition \ref{prop: high-ob-app}, we can obtain an observability inequality in $G$ up to a compact term.

Let $N$ be the same as in Proposition \ref{prop: high-ob-app}. For $(u_0,u_1)\in  \dot H^1(\T)\times L^2(\T)$, we decompose them into 
\begin{equation*}
u_0=P_Nu_0 +P^Nu_{0}, \quad u_1=P_Nu_1+P^Nu_{1},
\end{equation*}
where $P_N, P^N$ denote the low and high frequency projection, namely
$$
P_N\phi= \sum_{|k|\leq N}c_k \ee^{\ii k x} \quad \mbox{ if } \phi=\sum_{k\in \Z}c_k \ee^{\ii kx}
$$
and $P^N:= Id-P_N$.
If $u$ is a solution to \eqref{eq: wave-eq-0}, then $P^N u$ and $P_Nu$ are the solutions to \eqref{eq: wave-eq-0} with initial states $(P^Nu_0,P^Nu_1)$ and $(P_Nu_{0},P_Nu_{1})$, respectively. Therefore
\begin{align*}
\|\partial_x u_0\|^2_{L^2(\T)}+\|u_1\|^2_{L^2(\T)}=\|\partial_x(P^Nu_0)\|^2_{ L^2(\T)}+\|P^Nu_1\|^2_{L^2(\T)}+\|\partial_x(P_N u_{0})\|^2_{ L^2(\T)}+\|P_Nu_{1}\|^2_{L^2(\T)}.
\end{align*}
Using the high-frequency estimate \eqref{eq: high-est}, we have
\begin{equation*}
\|\partial_x u_0\|^2_{ L^2(\T)}+\|u_1\|^2_{L^2(\T)}\leq C\iint_G|\partial_t (P^N u)(t,x)|^2\d x\d t+\|\partial_x(P_N u_{0})\|^2_{ L^2(\T)}+\|P_Nu_{1}\|^2_{L^2(\T)}.
\end{equation*}
Due to the linear superposition of $u=P^Nu+P_Nu$, we derive that
\begin{multline*}
\|\partial_x u_0\|^2_{ L^2(\T)}+\|u_1\|^2_{L^2(\T)}\\
\leq 2C\iint_G|\partial_t u|^2\d x\d t+2C\iint_G|\partial_t (P_N u)|^2\d x\d t+\|\partial_x(P_N u_{0})\|^2_{ L^2(\T)}+\|P_N u_{1}\|^2_{L^2(\T)}.
\end{multline*}
Using the low-frequency truncation and the conservation of energy for wave equations, we deduce that
\begin{align*}
&2C\int_G|\partial_t (P_N u)(t,x)|^2\d x\d t+\|\partial_x(P_N u_{0})\|^2_{ L^2(\T)}+\|P_Nu_{1}\|^2_{L^2(\T)}\\
&\leq (2 CT + 1)\left(\|\partial_x(P_N u_{0})\|^2_{ L^2(\T)}+\|P_Nu_{1}\|^2_{L^2(\T)}\right)\\
&\leq (2CT+ 1)N^2\left(\|P_N u_{0}\|^2_{ \dot L^2(\T)}+\|P_Nu_{1}\|^2_{H^{-1}(\T)}\right).
\end{align*}
As a consequence, we obtain the weak observability up to a compact term:
\begin{equation*}
\|\partial_x u_0\|^2_{L^2(\T)}+\|u_1\|^2_{L^2(\T)}\leq C\left(\iint_G|\partial_t u|^2\d x\d t+\|P_N u_{0}\|^2_{\dot L^2(\T)}+\|P_Nu_{1}\|^2_{H^{-1}(\T)}\right).
\end{equation*}
This finishes the proof of Proposition \ref{prop: weak-wave-ob}.
\end{proof}
\vspace{1mm}

\section{A functional equivalence between two observability inequalities}\label{sec: HUM-app} 
In this appendix, we recall a standard observability inequality for the wave equation in a lower-regularity setting, and explain its equivalence with \eqref{eq: wave-ob}. 
This equivalence is classical and follows from a straightforward functional argument, which we include here for completeness.

There exists a constant $C>0$ such that  every solution of 
\begin{equation}\label{eq: wave-low}
(\partial_t^2-\partial_x^2)v=0,\;\;(v,\partial_t v)|_{t=0}=(v_0,v_1)\in L^2(\T)\times \dot H^{-1}(\T),
\end{equation}
satisfies 
\begin{equation}\label{eq: wave-ob-no-derivative}
\|v_0\|^2_{L^2(\T)}+\|v_1\|^2_{\dot H^{-1}(\T)}\leq C\iint_G|v(t,x)|^2\d t\d x.
\end{equation}

\begin{lemma}
The observability inequality \eqref{eq: wave-ob-no-derivative} for \eqref{eq: wave-low} is equivalent to the observability inequality \eqref{eq: wave-ob} for  \eqref{eq: wave-eq-0}.
\end{lemma}
\begin{proof}
We first note that  \eqref{eq: wave-ob} is equivalent to the same inequality restricted to initial data  $(u_0, u_1)\in \dot H^1(\T)\times L^2(\T)$.
Assume that \eqref{eq: wave-ob} holds. Given $(u_0, u_1)\in \dot H^1(\T)\times L^2(\T)$, define  $(v_0, v_1):= (u_1, \partial_x^2 u_0)\in L^2(\T)\times \dot H^{-1}(\T)$. Then the corresponding solution $v$ of \eqref{eq: wave-low} satisfies
\eqref{eq: wave-ob-no-derivative}. 
Conversely, assume that \eqref{eq: wave-ob-no-derivative} holds. For $(v_0, v_1)\in L^2(\T)\times \dot H^{-1}(\T)$, set $(u_0, u_1):= ((\partial_x^2)^{-1} v_1, v_0)\in \dot H^1(\T)\times L^2(\T)$. Then the corresponding solution $u$ of \eqref{eq: wave-eq-0} satisfies
\eqref{eq: wave-ob}.
\end{proof}

\vspace{2mm}

\subsection*{Acknowledgements}
The authors would like to thank Nicolas Burq, Jean-Michel Coron, Gengsheng Wang, and Yubiao Zhang for valuable discussions during the preparation of this manuscript.  Shengquan Xiang is partially supported by the NSFC under grant 12571474, and Ming Wang is partially supported by the NSFC under grant 12571260.

    \normalem
    \bibliographystyle{alpha}
    \bibliography{symmetrywave}

\end{document}